\documentclass[11pt,oneside]{amsart}
\usepackage[pagebackref,breaklinks,unicode]{hyperref} % Makes references and citations into links.
\title{Largest acylindrical actions of free-by-cyclic groups}

\author{Monika Kudlinska}
\address{DPMMS, Centre for Mathematical Sciences, Wilberforce Road, Cambridge,
CB3 0WB, UK, and Emmanuel College, St Andrew's Street, Cambridge CB2 3AP, UK}
\email{m.kudlinska@dpmms.cam.ac.uk}

\author{Harry Petyt}
\address{Mathematical Institute, University of Oxford, UK}
\email{harrypetyt@gmail.com}

\usepackage{amsmath, amssymb, amsthm} % Lots of maths functionality
\usepackage[english]{babel} % Language and hyphenation.
\usepackage[font=small,justification=centering]{caption,subcaption} % Subfigures
\usepackage[nodayofweek]{datetime}
\usepackage{enumitem} \setlist{nosep} % Change enumeration labelling with \begin{enumerate}[a)] etc.
\usepackage[T1]{fontenc} % For font encoding, to allow accents, copy & paste, inequality signs, etc. to all work nicely.
\usepackage[a4paper]{geometry} % Changing page shape
\usepackage[utf8]{inputenc} % To be loaded after fontenc, also for encoding.
    \usepackage{csquotes} % Quotation environment
\usepackage{ifthen} % For \numberedtheorem
\usepackage{mathabx} % Contains \pnest symbol and more. Clashes with accents for \ring
\usepackage{mathtools} % Uses amsmath, fixes quirks and adds functionality.
\usepackage[dvipsnames]{xcolor} % Allow links to have colour. Needs to be before hyperref and tikz-cd
\usepackage{setspace} % Allows \onehalfspacing etc. for changing gaps between lines
\usepackage{textcomp} % Needed to have \texttildelow, for copyable citations
\usepackage{tikz-cd} % Commutative diagrams
\usepackage[capitalize]{cleveref}
\usepackage{xfrac}
\usepackage{MnSymbol}

\makeatletter
\newcommand*{\acts}{\mathrel{%
    \begin{tikzpicture}[x=\f@size pt, y=\f@size pt]
        \draw[->] (0,0) arc[start angle=180, x radius=.4, y radius=.6, end angle=0];
    \end{tikzpicture}%
}}
\makeatother

\let\OLDthebibliography\thebibliography \renewcommand\thebibliography[1]{   % Reduce bibliography spacing
    \OLDthebibliography{#1}\setlength{\parskip}{0pt}\setlength{\itemsep}{0pt plus 0.3ex}} 

\makeatletter\def\subsection{\@startsection{subsection}{1}\z@{.7\linespacing\@plus\linespacing}
    {.5\linespacing}{\normalfont\scshape\centering}}\makeatother % centrally aligned subsections

\renewcommand*{\backrefalt}[4]{\ifcase #1 (Not cited).\or (Cited p.~#2).\else (Cited pp.~#2).\fi} % Cited on...

\newcounter{shcount}
\newcounter{thmcount}
\newcounter{enumlabelcount}
\newcounter{claimcount}

\newcommand*{\numberedtheorem}[3]{\theoremstyle{plain}\newtheorem*{makethm\thethmcount}{#1}
    \ifthenelse{\equal{#2}{}}{\begin{makethm\thethmcount}#3\end{makethm\thethmcount}\stepcounter{thmcount}}
    {\begin{makethm\thethmcount}[#2]#3\end{makethm\thethmcount}\stepcounter{thmcount}}} % Fixed-number thms

\makeatletter\newcommand\enumlabel[1][]{\item[#1]
    \refstepcounter{enumlabelcount}\def\@currentlabel{#1}}\makeatother

\newenvironment{claim}{\refstepcounter{claimcount}\medskip\noindent
    \textbf{Claim \theclaimcount:}\hspace{0.5mm}}{}
\newenvironment{claim*}{\medskip\noindent\textbf{Claim:}\hspace{0.5mm}}{}
\newenvironment{claimproof}{\medskip\noindent\emph{Proof of Claim \theclaimcount.}\hspace{0.5mm}}
    {\leavevmode\unskip\penalty9999\hbox{}\nobreak\hfill\quad\hbox{$\diamondsuit$}\medskip}
\newenvironment{claim*proof}{\medskip\noindent\emph{Proof of Claim.}\hspace{0.5mm}}
    {\leavevmode\unskip\penalty9999\hbox{}\nobreak\hfill\quad\hbox{$\diamondsuit$}\medskip}

\newcommand*{\AH}{\ensuremath{\mathcal{AH}}}

\newcommand*{\eps}{\varepsilon}

\newcommand*{\R}{\mathbf{R}}

\newcommand*{\Z}{\mathbf{Z}}

\newcommand*{\cal}{\mathcal}

\newcommand{\caly}{\mathcal{Y}}
\newcommand{\calx}{\mathcal{X}}
\newcommand{\calh}{\mathcal{H}}

\newcommand*{\ssm}{\smallsetminus}

\DeclareMathOperator{\Aut}{Aut}
\DeclareMathOperator{\Cay}{Cay}
\DeclareMathOperator{\diam}{diam}
\DeclareMathOperator{\dist}{\mathsf{d}}
\DeclareMathOperator{\GL}{GL}

\DeclareMathOperator{\Out}{Out}

\newcommand{\ignore}[2]{\left\{\kern-.7ex\left\{#1\right\}\kern-.7ex\right\}_{#2}}

\newcommand*{\sgen}[1]{\langle#1\rangle}

\definecolor{harrycomment}{rgb}{0.6,0,0.4}

\definecolor{monikacomment}{rgb}{0.2,0,0.8}

\newtheorem{question}{Question}
\swapnumbers
\newtheorem{theorem}{Theorem}[section]
\newtheorem{lemma}[theorem]{Lemma} 
\newtheorem{corollary}[theorem]{Corollary}
\newtheorem{proposition}[theorem]{Proposition}
\theoremstyle{definition}
\newtheorem{definition}[theorem]{Definition}
\newtheorem{remark}[theorem]{Remark}

\begin{document}

\begin{abstract}
%We prove that the relative Cayley graph with respect to maximal product subgroups of a \{finitely generated free\}-by-cyclic group with polynomially growing monodromy is a quasitree, and the regular action is acylindrical. 

We show that every finitely generated free-by-cyclic group $G$ admits a largest acylindrical action on a hyperbolic space $X$ obtained by coning off maximal product subgroups of $G$. We characterise Morse geodesics of $G$ as those that project to quasigeodesics in $X$, thus showing that all finitely generated free-by-cyclic groups are Morse local-to-global. We also characterise the stable and strongly quasiconvex subgroups of $G$. Finally, we compute the Morse boundary for \{finitely generated free\}-by-cyclic groups with unipotent and polynomially growing monodromy.
\end{abstract}
\maketitle

%%%%%%%%%%%%%%%%%%%%%%%%%%%%%%%%%%%%%%%%%%%%%%%%%%
\section{Introduction}

A guiding principle in geometric group theory is to study groups through their actions on non-positively curved spaces. This approach has seen particular success in the setting of \emph{Gromov hyperbolic groups}; that is, groups that act properly and cocompactly on a Gromov hyperbolic space $X$. In that case, the Cayley graph of the group with respect to any finite generating set is quasi-isometric to $X$, and thus the geometry of the group is fully reflected in the hyperbolic structure of $X$. 

Hyperbolicity places strong restrictions on the geometry and algebra of a group. In particular, the Cayley graph of a Gromov hyperbolic group with respect to a finite generating set does not admit quasiisometrically embedded product regions (when we say a space or group \emph{is a product}, we mean it is a direct product of two unbounded spaces or groups). 

In some cases, existence of product regions is the only obstruction to hyperbolicity: after ``coning off'' an equivariant collection of product regions in the Cayley graph $\Gamma = \mathrm{Cay}(G; S)$, one obtains a Gromov hyperbolic space $\hat{\Gamma}$. While the resulting action of $G$ on $\hat{\Gamma}$ fails to be proper, establishing that it is non-elementary and \emph{acylindrical} (a weakening of properness) is sufficient to recover many interesting properties of the group \cite{dahmaniguirardelosin:hyperbolically,osin:acylindrically}. In particular, such an action will necessarily give rise to a ``largest'' (in the sense of \cite{abbottbalasubramanyaosin:hyperbolic}) and thus canonical acylindrical action, resulting in the group being $\mathcal{AH}$-accessible (see \cref{defn:accessibility}). 

Examples of this phenomenon include the mapping class groups, where the corresponding largest acylindrical action is on the curve graph \cite{bowditch:tight}, right-angled Artin groups, where the largest action is on the extension graph \cite{kimkoberda:embedability,kimkoberda:geometry}, and more generally all hierarchically hyperbolic groups \cite{abbottbehrstockdurham:largest}.

Moreover, the space $\hat{\Gamma}$ can often be used to detect subgroups of $G$ that satisfy various generalisations of quasiconvexity. For instance, in all of the above cases, a subgroup $H \leq G$ is \emph{stable} in the sense of \cite{durhamtaylor:convex}, if and only if it quasiisometrically embeds in $\hat{\Gamma}$. In the case of mapping class groups, stable subgroups are exactly the convex cocompact subgroups \cite{farbmosher:convex,kentleininger:shadows,hamenstadt:word}, whilst in right-angled Artin groups a subgroup is stable if and only if it is purely loxodromic \cite{koberdamangahastaylor:geometry,kimkoberda:geometry}.

In this paper, we show that precisely the same phenomenon occurs in the class of finitely generated \emph{free-by-cyclic} groups. A group $G$ is \emph{free-by-cyclic} if it fits into a short exact sequence $1\to F\to G\to\Z\to 1$ with $F$ a free group. If $F$ can be taken to be finitely generated, then we say $G$ is \emph{\{f.g.\ free\}-by-cyclic}. The \emph{monodromy map} is the corresponding action of $\Z$ on $F$. We prove the following: % The group $G$ is \emph{\{finitely generated free\}-by-cyclic} if $F$ can be taken to be finitely generated. The \emph{monodromy map} is the corresponding action of $\Z$ on $F$. 

\begin{theorem} \label{thm:intro}
Let $G$ be a free-by-cyclic group with a finite generating set $S$, and let $\mathcal{P}$ be the collection of maximal product subgroups. The Cayley graph $X = \mathrm{Cay}(G; S \cup \mathcal{P})$ is hyperbolic, and it is unbounded if $G$ is not virtually a product. The regular action of $G$ on $X$ is the largest acylindrical action of $G$. A quasigeodesic of $G$ is Morse if and only if it is a quasigeodesic in $X$.

Moreover, if $G$ is \{f.g.\ free\}-by-cyclic with polynomially-growing monodromy, then $X = \mathrm{Cay}(G; S \cup \mathcal{P})$ is a quasitree. 
\end{theorem}

By now much is known about the geometry of finitely generated free-by-cyclic groups \cite{brinkmann:hyperbolic,bridsongroves:quadratic,hagenwise:cubulating:general, bongiovannighoshgultepehagen:characterizing}. In particular, they are hyperbolic relative to their collection of maximal \{f.g.\ free\}-by-cyclic subgroups with polynomially-growing monodromies \cite{ghosh:relative, linton:geometry}.  Since the latter are \emph{thick} \cite{hagen:remark} and thus do not admit proper relatively hyperbolic structures \cite{behrstockdrutumosher:thick}, such a relatively hyperbolic structure is maximal (and hence also canonical). 

It follows that many aspects of the geometry of a finitely generated free-by-cyclic group can be understood in terms of its free-by-cyclic subgroups with polynomially-growing monodromy. \cref{thm:intro} extends this by establishing a canonical maximal \emph{acylindrically hyperbolic} structure, which allows a further reduction to subgroups with \emph{linearly}-growing monodromy. We exploit this idea in the proof of \cref{thm:intro:stable} below. 

Our result helps explain the recently observed phenomenon that obstructions to non-positive curvature in free-by-cyclic groups are often at the level of linear growth \cite{munropetyt:obstructions,bongiovannighoshgultepehagen:characterizing}. 
Note that it was already known that finitely generated free-by-cyclic groups are acylindrically hyperbolic by combining \cite{genevoishorbez:acylindrical} with \cite{linton:geometry}, though our arguments do not rely on \cite{genevoishorbez:acylindrical}.

\medskip

We obtain several corollaries of our construction. First, the statement about Morse quasigeodesics in \cref{thm:intro} implies that $X$ is a \emph{Morse detectability space} in the sense of \cite{russellsprianotran:local}. In particular, we answer a question of Russell--Spriano--Tran \cite[Question~2]{russellsprianotran:local}:

\begin{corollary} \label{cor:intro}
Every finitely generated free-by-cyclic group is Morse local-to-global.
\end{corollary}

As in the case of mapping class groups and right-angled Artin groups, the space $X$ allows us to characterise subgroups of finitely generated free-by-cyclic groups that satisfy various notions of quasiconvexity.

A finitely generated subgroup $H$ of a finitely generated group $G$ is \emph{stable} if it is undistorted and every quasigeodesic in $G$ with endpoints in $H$ is Morse \cite{durhamtaylor:convex}. Following \cite{tran:onstrongly}, a subgroup $H$ of a finitely generated group $G$ is said to be \emph{strongly quasiconvex} (sometimes also called \emph{Morse} in the literature \cite{genevois:hyperbolicities}) if there exists some $M = M(S) >0$ such that every geodesic in $\mathrm{Cay}(G;S)$ joining two points in $H$ is contained in the $M$-neighbourhood of $H$.

\begin{theorem} \label{thm:intro:stable}
Let $G$ be a finitely generated free-by-cyclic group. Let $\cal P$ be the collection of maximal product subgroups of $G$ and let $X$ be the hyperbolic space from \cref{thm:intro}. For each finitely generated $H<G$, the following are equivalent.
\begin{enumerate}
\item   $H$ is stable in $G$.
\item   $H$ intersects each subgroup in $\cal P$ trivially.
\item   The orbit map of $H$ on $X$ is a quasiisometric embedding.
\end{enumerate}
Moreover, if $G$ if \{f.g.\ free\}-by-cyclic with polynomially growing monodromy and $H$ has infinite index in $G$, then $H$ is strongly quasiconvex if and only if $H$ is stable.
\end{theorem}

In particular, we observe that $X$ is a \emph{universal recognising space} for $G$ in the sense of \cite{balasubramanyachesserkerrmangahastrin:non}; see Section~\ref{sec:stability}.

Combined with \cite[Corollary~6.8]{tran:onstrongly}, the statement in Theorem~\ref{thm:intro:stable} about strongly quasiconvex subgroups answers a question of I.~Kapovich \cite[Question~1.6]{aimpl:rigidity}; see \cref{cor:general-Morse-subgp}.

\medskip

Our final application of \cref{thm:intro} is to the study of Morse boundaries of \{f.g.\ free\}-by-cyclic groups with polynomially growing monodromies. An \emph{$\omega$-Cantor} space is roughly the direct limit of countably many Cantor spaces. We show:

\begin{theorem} \label{thm:intrp:boundary}
Let $G$ be a \{f.g.\ free\}-by-cyclic group with polynomially growing monodromy. The Morse boundary of $G$ is an $\omega$-Cantor space. 
\end{theorem}

\subsection{Further questions}

\cref{thm:intro} characterises the Morse geodesics of a finitely generated free-by-cyclic groups as those that project to quasigeodesics in a particular hyperbolic space $X$. It is known that for each geodesic space $Y$ there is a hyperbolic space $Z$ that characterises the \emph{strongly contracting} geodesics of $Y$ in a similar way \cite{petytzalloum:constructing,zbinden:hyperbolic}. It is natural to ask whether the spaces $X$ and $Z$ are the same for finitely generated free-by-cyclic groups. In particular:

\begin{question}
Are Morse geodesics in finitely generated free-by-cyclic groups strongly contracting in some fixed Cayley graph?
\end{question}

In \cite{bongiovannighoshgultepehagen:characterizing}, it is shown that a finitely generated free-by-cyclic group is \emph{hierarchically hyperbolic} provided it satisfies a certain condition on the intersections of its product subgroups. Hierarchically hyperbolic groups always admit largest acylindrical actions on hyperbolic spaces \cite{abbottbehrstockdurham:largest}, so \cref{thm:intro} and \cref{cor:intro} show that some features of hierarchical hyperbolicity are manifest even in the absence of the property itself. That said, our methods are considerably more direct than going via hierarchical hyperbolicity. Nevertheless, this leads to the following question.

\begin{question}
Is every finitely generated free-by-cyclic group hierarchically hyperbolic relative to its \{f.g.\ free\}-by-cyclic subgroups with linearly growing monodromies?
\end{question}

Since every finitely generated free-by-cyclic group is hyperbolic relative to its \{f.g.\ free\}-by-cyclic subgroups with polynomially growing monodromy, \cref{thm:intrp:boundary} raises the following.

\begin{question}
What is the Morse boundary of a general finitely generated free-by-cyclic group? 
\end{question}

\subsection*{Acknowledgments}

The first author thanks Alice Kerr, Hoang Thanh Nguyen and Stefanie Zbinden for interesting discussions. The second author thanks the Isaac Newton Institute for Mathematical Sciences for support and hospitality during the program \textit{Operators, Graphs, and Groups}, where this project started; EPSRC grant number EP/Z000580/1.

%%%%%%%%%%%%%%%%%%%%%%%%%%%%%%%%%%%%%%%%%%%%%%%%%%
\section{Preliminaries}

\subsection{Bass--Serre theory}

For the purposes of this article, we define a \emph{graph} $X$ to be a 1-dimensional CW-complex with the set of $0$-cells denoted by $V(X)$ and the set of $1$-cells denoted by $E(X)$. An \emph{orientation} on $X$ is a map $ E(X) \to V(X)$ that sends each edge $e \in E(X)$ to one of its endpoints $e^{-}$, which we declare to be the \emph{initial endpoint}. We let $e^{+}$ be the other endpoint of $e$, which we call the \emph{terminal endpoint}. A \emph{tree} is a connected graph with no cycles.

A \emph{graph of groups} $\mathbb{G}$ consists of an oriented connected graph $X$ and an assignment of a group $G_v$ to each vertex $v \in V(X)$ and a group $G_e$ to each edge $e \in E(X)$. Furthermore, for each edge $e \in E(X)$, there is a pair of monomorphisms 
\[\partial_{e}^{+} \colon G_e \to G_{e^{+}} \text{ and } \partial_{e}^{-} \colon G_e \to G_{e^{-}}.\] 

The \emph{Bass group} associated to $\mathbb{G}$ is
\[\mathrm{Bass}(\mathbb{G}) = \left. \bigast_{v\in V(X)} G_v \ast F(t_e)_{e \in E(X)} \middle/ \llangle t_{e}^{-1} \partial_{e}^{-}(g) t_e = \partial_e^{+}(g) \, \forall g \in G_e  \forall e \in E(X)\rrangle \right..\]

Given a choice of a spanning tree $\mathcal{T}$ of $X$, the \emph{fundamental group} of the graph of groups $\pi_1(\mathbb{G}, \mathcal{T})$ is defined to be 
\[\pi_1(\mathbb{G}, \mathcal{T}) = \left.\mathrm{Bass}(\mathbb{G}) \middle/ \llangle t_e \in E(\mathcal{T}) \rrangle \right. .\]

The \emph{universal cover} relative to the fundamental group of the graph of groups $\pi_1(\mathbb{G}, \mathcal{T})$ is a tree $T$ where the vertices and edges are labelled by cosets of vertex and edge groups, 
\[ V(T) = \bigsqcup_{v\in V(X)} \pi_1(\mathbb{G}, \mathcal{T}) / G_v \text{ and }E(T) =  \bigsqcup_{e\in E(X)} \pi_1(\mathbb{G}, \mathcal{T}) / G_e,\]
and adjacency is determined by inclusion. We call $T$ the \emph{Bass--Serre} tree associated to the graph of groups $\pi_1(\mathbb{G}, \mathcal{T})$. There is a natural left action of $G = \pi_1(\mathbb{G}, \mathcal{T})$ on $T$ determined by the action of $G$ on the left cosets of the vertex and edge groups. 

We say that a group $G$ \emph{admits a graph-of-groups splitting} if there is an isomorphism $G \cong \pi_1(\mathbb{G}, \mathcal{T})$ for some graph of groups $\mathbb{G}$ and maximal spanning tree $\mathcal{T}$. We will sometimes omit the spanning tree from the notation, if it is not relevant.

%%%%%%%%%%
\subsection{Dynamics of outer automorphisms}

Let $F$ be a free group of finite rank and fix a free basis $S$ of $F$. For any $g \in F$, let $\|\bar{g} \|_S$ denote the word length of a shortest representative in the conjugacy class $\bar{g}$ of $g$. Let $\Phi\in\Out(F)$. We say that $\bar{g}$ \emph{grows polynomially of degree $d \in \mathbb{N}$} under the iterations of $\Phi$ if there exist constants $A, B > 0$ such that for all $n \in \mathbb{N}$,
\[ A n^d - A \leq \| \Phi^n(\bar{g}) \|_S \leq B n^d + B.\]

An outer automorphism $\Phi \in \Out(F)$ \emph{grows polynomially of degree at most $d$} if every $\bar{g}$ grows polynomially of degree at most $d$. We say that $\Phi$ is \emph{unipotent} if it induces a unipotent element of $\GL(\operatorname{rank}F ;\Z)$ in the natural quotient given by abelianising $F$. We will abbreviate (unipotent and) polynomially growing to \emph{(U)PG}. 

\medskip

Let $\Phi \in \Out(F)$ be an outer automorphism. A \emph{topological representative} of $\Phi$ is a pair $(f, \Gamma)$ where $\Gamma$ is a connected oriented graph with $\pi_1(\Gamma) \cong F$ and $f \colon \Gamma \to \Gamma$ is a homotopy equivalence that determines $\Phi$. Moreover, we have that $f(V(\Gamma)) \subseteq V(\Gamma)$ and $f$ is locally injective on the interiors of edges. 

Let $(f, \Gamma)$ be a topological representative with a maximal filtration
\[ \emptyset = \Gamma_0 \subseteq \Gamma_1 \subseteq  \ldots \subseteq \Gamma_n = \Gamma\]
and suppose that the following conditions are satisfied for each index $i$: 
\begin{enumerate}
    \item $f(\Gamma_i) \subseteq \Gamma_i$; 
    \item $\Gamma_{i+1}$ is the union of $\Gamma_i$ and an edge $e_{i+1}$; 
    \item \label{it:loop} $f(e_i) = e_i \rho_i$ where $\rho_i$ is an immersed loop in $\Gamma_{i - 1}$, called the \emph{suffix} of $f(e_i)$; 
    \item for any two distinct edges $e_i$ and $e_j$, if the corresponding suffixes $\rho_i$ and $\rho_j$ are non-trivial then $\rho_i \neq \rho_j$. 
\end{enumerate}
\smallskip 

Note that by \cref{it:loop}, every vertex in $\Gamma$ is fixed by $f$.

A \emph{Nielsen path} for $(f, \Gamma)$ is a path $\alpha$ in $\Gamma$ such that $f(\alpha)$ is homotopic to $\alpha$ relative endpoints. A \emph{unipotent representative} \cite[Definition~3.13]{bestvinafeighnhandel:tits:2} is a topological representative $(f, \Gamma)$ as above, with extra conditions that allow for greater control of Nielsen paths. We will omit the definition as it will not be relevant to this article. Instead, we note the following:

\begin{proposition}[{\cite[Theorem~5.1.8]{bestvinafeighnhandel:tits:1}}]
    If $\Phi \in \mathrm{Out}(F)$ is UPG then it admits a unipotent representative.
\end{proposition}

%In \cite{bestvinafeighnhandel:tits:1}, Bestvina--Feighn--Handel defined a class of topological representatives with particularly good properties. These were further studied in \cite{bestvinafeighnhandel:tits:2} in the case of polynomially-growing outer automorphisms, where they were called \emph{unipotent representatives}. We will omit the full definition and instead discuss the relevant features of a unipotent representative $(f, \Gamma)$, fixing the relevant notation as we go. 

\subsection{Free-by-cyclic groups}

A group $G$ is \emph{free-by-cyclic} if it contains a normal subgroup $F \trianglelefteq G$ that is free and such that $G / F \cong \Z$. Moreover, $G$ is \emph{\{f.g.\ free\}-by-cyclic} if there exists a finitely generated such $F$. After picking a lift of a generator of $\Z$ in $G$, we may express $G$ as an (inner) semidirect product $G  = F \rtimes_{\phi} \Z$ where $\phi \in \Aut(F)$. We call $\phi$ the \emph{monodromy} of the splitting. Note that any two lifts of the generator of $\mathbb{Z}$ correspond to monodromies that differ by an inner automorphism of $F$. Conversely, for any two elements $\phi, \psi \in \Aut(F)$ that differ by an inner automorphism, there is an isomorphism $F \rtimes_{\phi} \Z \to F' \rtimes_{\psi} \Z$ mapping $F \mapsto F'$. Thus, we will also refer the the outer class $[\phi] \in \Out(F)$ as the monodromy. 

 % \begin{lemma}
 % If $F$ and $F'$ are finitely generated and $\phi \in \Aut(F)$ and $\psi \in \Aut(F')$ are such that
 % \[ F \rtimes_{\phi} \Z \cong F' \rtimes_{\psi} \Z,\]
 % then the outer classes $[\phi] \in \Out(F)$ and $[\psi] \in \Out(F')$ are both exponentially growing or both polynomially growing of the same degree.
 % \end{lemma}

\begin{proposition}\label{prop:rel-hyp}
Every finitely generated free-by-cyclic group is hyperbolic relative to a collection $\mathcal{H}$ of maximal \{f.g.\ free\}-by-cyclic subgroups with PG monodromy.  
\end{proposition}

\begin{proof}
Every finitely generated free-by-cyclic group is hyperbolic relative to a collection of \{f.g.\ free\}-by-cyclic subgroups by \cite[Theorem~6.1]{linton:geometry}. Moreover, every \{f.g.\ free\}-by-cyclic group is hyperbolic relative to a  collection of \{f.g.\ free\}-by-cyclic subgroups with polynomially growing monodromies \cite{ghosh:relative, dahmanili:relative}. Combining these yields the claim.
\end{proof}

 \subsection{Splittings of UPG free-by-cyclic groups}\label{sec:splittings}

 Throughout this subsection we will assume that $F$ is a finitely generated free group. We begin by recalling the splitting of \{f.g.\ free\}-by-cyclic groups with linearly growing monodromies, which was constructed independently by Andrew--Martino in \cite{andrewmartino:free-by-cyclic} and Dahmani--Touikan in \cite{dahmanitouikan:unipotent}, who moreover showed that the splitting is acylindrical \cite[Proposition~3.2]{dahmanitouikan:unipotent}.

\begin{proposition}\label{prop:linear-splitting}
    Let $\phi \in \Aut(F)$ be a representative of a unipotent and linearly growing outer automorphism. Then $G = F\rtimes_{\phi} \Z$ admits a graph-of-groups splitting $G \cong \pi_1(\mathbb{G})$ where the underlying graph $(X, V_0(X), V_1(X))$ is bipartite, and such that the vertex groups $G_u$ for $u \in V_0(X)$ are non-abelian and maximal product subgroups, the vertex groups $G_v$ for $v \in V_1(X)$ are maximal $\Z^2$-subgroups, the edge groups are $\Z^2$-subgroups and all the edge inclusions to vertex groups in $V_1(X)$ are onto. Moreover, the splitting is 4-acylindrical.
\end{proposition}

Suppose now that $\phi \in \Aut(F)$ represents a unipotent outer automorphism with polynomial growth of degree $d>1$. Then $G = F \rtimes_{\phi} \Z$ admits a splitting over maximal cyclic subgroups arising from the topmost edges of the unipotent representative of the outer class $\Phi = [\phi]$ (see e.g. \cite{ macura:detour, hagen:remark, andrewhugheskudlinska:torsion}). We briefly recall the construction below, as it is relevant for this article.

\smallskip

Let $(f, \Gamma)$ be a unipotent representative. The \emph{mapping torus} of $f$ is the quotient of $\Gamma \times [0,1]$ by the equivalence relation $(x,0) \sim (f(x),1)$ for all $x \in \Gamma$. Since $f$ fixes the vertices of $\Gamma$, the fibre over any vertex $v$ becomes a closed loop in $M_f$, which we call the \emph{vertical loop} at $v$.

Identifying $\Gamma$ with the equivalence class of $\Gamma \times \{0\}$ in $M_f$, we choose a basepoint $x_0 \in V(\Gamma)$ and identify $G$ with $\pi_1(M_f, x_0)$. Let $F \leq G$ be the subgroup corresponding to $\pi_1(\Gamma, x_0)$ and let $t \in G$ be the class of the vertical loop at $x_0$. This yields the semidirect product structure $G = F \rtimes_{\phi} \langle t \rangle$.

Now let $ \mathcal{E}_d = \{e_{i_1}, \ldots, e_{i_l}\}$ be the edges of $\Gamma$ that grow polynomially of degree $d$ under the iteration of $f$. Let $\{\Gamma^{(j)}\}_{j\in J}$ be the components of $\Gamma \setminus \bigcup_{e \in \mathcal{E}_d} \mathrm{int}(e)$. Then $f$ preserves each component $\Gamma^{(j)}$ and the suffix of $f(e_{i_k})$ is entirely contained in the component that contains the terminal vertex $t(e_{i_k})$ of $e_{i_k}$. 

We obtain a graph-of-groups splitting $G \cong \pi_1(\mathbb{G})$ by collapsing each $\Gamma^{(j)}$ to a point, and taking the vertex groups to be the fundamental groups of mapping tori of $f$ restricted to each $\Gamma^{(j)}$. Hence, each vertex group $G_v$ is of the form 
\[ G_v = F_v \rtimes_{\phi_v} \langle t_v \rangle,\]
where $F_v$ is identified with a (possibly trivial) free factor of $F$ and $t_v$ is identified with an element of $ F \cdot t$. For each topmost edge $e \in \mathcal{E}_d$, there is an edge group $G_e = \langle t_e \rangle$ where the image of $t_e$ in $G$ is contained in the coset $F \cdot t$. We call the splitting $G \cong \pi_1(\mathbb{G})$ the \emph{topmost edge splitting}. It was shown in \cite[Lemma~5.2]{kudlinskavaliunas:free} that the splitting is 2-acylindrical.

%The corresponding Bass--Serre tree $T$ can be identified with the universal cover $\tilde{\Gamma}$ of $\Gamma$ after collapsing all lifts of subgraphs $\Gamma^{(j)}$. The homotopy equivalence $f \colon \Gamma \to \Gamma$ induces an identity map of $X_{\mathbb{G}}$ and thus lifts to an automorphism $\tilde{f} \colon T \to T$. 

Note that for each vertex $v$, either $G_v$ is infinite cyclic or $F_v \neq 1$ and the automorphism $\phi_v \in \Aut(F_v)$ represents a UPG element of $\Out(F_v)$ with growth of degree $d' \leq d-1$. If $d' > 1$, then $G_v$ itself admits a topmost edge splitting. We call the resulting hierarchy the \emph{cyclic hierarchy} associated to topmost edge splittings. If $d' = 1$ then $G_v$ admits an acylindrical graph of groups splitting coming from \cref{prop:linear-splitting}.

Let $e \in \mathcal{E}_d$ and let $v = e^{+}$. Then $f(e) = e \rho$ for some immersed path $\rho$, and thus the edge inclusion $\partial_e^{+} \colon G_e \to G_{v}$ is such that $\partial_e^{+}(t_e) = t_{v} g_{\rho}^{-1}$ where the conjugacy class of $g_{\rho} \in F_v$ grows polynomially of degree $d-1$ under the outer class of $\phi_v$. It follows that $t_{v} g_{\rho}^{-1}$ acts loxodromically on the Bass--Serre tree of $G_{v}$ (see \cite[Remark~4.6]{bongiovannighoshgultepehagen:characterizing}).

% \begin{lemma}\label{lemma:distinct-images}

% Let $(f, \Gamma)$ be a train track that grows polynomially of degree $d > 1$ and let $e_1$ and $e_2$ be distinct topmost edges of $\Gamma$. Suppose that $e_1^{+}$ and $e_2^{+}$ or $e_2^{-}$ are contained in the same connected component $\Gamma'$ of $\Gamma \setminus \bigcup_{e \in \mathcal{E}_d} \mathrm{int}(e)$. Let $G_v$ be the vertex group corresponding to the component $\Gamma'$ and let  $t_{e_1}$ and $t_{e_2}$ be the generators of the edge groups $G_{e_1}$ and $G_{e_2}$, corresponding to $e_1$ and $e_2$, respectively. Then, $\partial_{e_1}^{+}(t_{e_1}) \neq \partial_{e_2}^{\pm }(t_{e_2})$.
% \end{lemma}

% \begin{proof}
%     Let $x_0$ be the vertex of $\Gamma'$ which coincides with $e_1^{+}$. Let $f'$ be the restriction of $f$ to $\Gamma'$. Then, we may identify $G_v \cong \pi_1(M_{f'}, x_0) \cong F_v \rtimes \langle t_v \rangle$ where $t_v \in G_v$ is the vertical loop at $x_0$ and $F_v = \pi_1(\Gamma', x_0)$. 

%     Since $f(e_1) = e_1 \rho$ for some immersed loop $\rho$ in $\Gamma'$ based at $x_0$, it follows that $\partial_{e_1}^{+}(t_{e_1}) = t_{v}g_{\rho}^{-1}$ for some $g_{\rho} \in F_v$. 
% \end{proof}

\smallskip

We end this section with the following useful fact. 

\begin{lemma}\label{lemma:common-powers}
    Let $G = F \rtimes_{\phi} \langle t \rangle$ be a \{f.g.\ free\}-by-cyclic group with UPG monodromy (not necessarily $\phi$). Let $g,h \in F \cdot t$ be distinct elements. Then $g$ and $h$ do not have a common non-trivial power. 
\end{lemma}

\begin{proof}
    Let $g,h \in F \cdot t$ be distinct. If $g$ and $h$ have a common non-trivial power then it must be the case that $g^n = h^n$ for some $n \in \mathbb{Z} \setminus 0$. However, a free-by-cyclic group with UPG monodromy has the unique roots property by \cite[Lemma~2.14]{dahmanihugheskudlinskatouikan:conjugacy}, and thus $g = h$. This is a contradiction.
\end{proof}

\subsection{Acylindrical actions on trees}

Let $\kappa$ be a non-negative integer. An action of $G$ on a simplicial tree $T$ is \emph{$\kappa$-acylindrical} if the pointwise stabiliser of any edge path in $T$ of length $\geq \kappa+1$ is trivial, where the length of an edge path is the number of edges.

\begin{lemma}\label{lemma:finite-extension-acylindrical}
    Let $G$ be a torsion-free group and $G' \leq_{f.i.} G$ a subgroup of finite index. Suppose that $G$ acts on a tree $T$ and the induced action of $G'$ on $T$ is $\kappa$-acylindrical. Then the action of $G$ on $T$ is $\kappa$-acylindrical.
\end{lemma}

\begin{proof}
    Let $\rho$ be an edge path in $T$ of length $\kappa +1$. Suppose that there exists an element $g \in G \setminus 1$ such that $g$ fixes $\rho$ pointwise. It follows that for each $n \in \mathbb{N}$, $g^n$ fixes $\rho$ pointwise. Let $M \in \mathbb{N}$ be such that $g^M \in G'$. Since $G$ is torsion free, $g^M$ is non-trivial. But this contradicts $\kappa$-acylindricity of the action of $G'$ on $T$.
\end{proof}

Let $G$ be a torsion-free group and suppose that $G$ admits a $\kappa$-acylindrical action on a tree $T$. Let $g,h \in G$ be loxodromic elements for the action on $T$ with axes $\alpha_g$ and $\alpha_h$ and translation lengths $\lambda_T(g)$ and $\lambda_T(h)$, respectively.

\begin{lemma}\label{lemma:tree-axes-intersections}
If $g$ and $h$ do not have a common non-trivial power, then 
\[\mathrm{diam}_T(\alpha_g \cap \alpha_h) \leq \kappa + \lambda_T(g) \cdot \lambda_T(h).\]
\end{lemma}

\begin{proof}
Let $I=\alpha_g\cap\alpha_h$. If $\diam I>\lambda_T(g)\lambda_T(h)+\kappa$, then there is a subinterval $J\subset I$ of length $\kappa+1$ such that every vertex $x\in J$ has $g^{\lambda_T(h)}x\in I$. Up to inverting $h$, we also have $h^{\lambda_T(g)}x=g^{\lambda_T(h)}x\in I$ for every $x\in J$, so the element $g'=g^{\lambda_T(h)}h^{-\lambda_T(g)}$ fixes $J$ pointwise. Since the action of $G$ on $T$ is $\kappa$-acylindrical, this implies that $g'=1$, and thus $h^{\lambda_T(g)} = g^{\lambda_T(h)}$. This is a contradiction. %and so $\mathrm{Axis}_T(h)=\mathrm{Axis}_T(h^{\lambda_T(g)})=\mathrm{Axis}_T(g)$. This is a contradiction.
\end{proof}

The next lemma bounds the coarse intersections of axes for $g$ and $h$ in actions of $G$ that dominate the acylindrical action on $T$. Let $X$ be a graph with an isometric $G$-action. Suppose that there is a 1-Lipschitz $G$-equivariant map $\pi \colon X \to T$. Pick a vertex $x_0 \in \pi^{-1}(\alpha_g)$ such that 
\[
\dist(g\cdot x_0, x_0) \,=\, \mathrm{min}\{\dist(g\cdot x, x) \mid x \in \pi^{-1}(\alpha_g)\} \,=:\, \lambda_X(g).
\] 
Pick $y_0 \in \pi^{-1}(\alpha_h)$. Let $A(g) = \{g^n \cdot x_0\}_{n\in \mathbb{Z}}$ and $A(h) = \{h^n \cdot y_0\}_{n \in \mathbb{Z}}$. 

\begin{lemma}\label{lemma:general-axes-intersections} 
Suppose that $g$ and $h$ are loxodromic for the $\kappa$-acylindrical action of $G$ on $T$ and do not have a common non-trivial power. For each $\epsilon > 0$, there exists a constant $K$ that depends only on $\epsilon$, $\kappa$, and the conjugacy classes of $g$ and $h$, such that
    \[ \mathrm{diam}_X(N_{\epsilon}(A(g)) \cap A(h)) \leq K.\]
\end{lemma}

\begin{proof}
Let $K = \frac{\lambda_X(g)}{\lambda_T(g)}(\kappa+\lambda_T(g)\lambda_T(h) +2\epsilon)$
and suppose that $ \mathrm{diam}(N_{\epsilon}(A(g)) \cap A(h)) > K$.  

After replacing $y_0$ with $h^k \cdot y_0$ and $x_0$ with $g^l \cdot x_0$ for some $k,l \in \Z$, we may assume that $d(x_0,y_0 )\leq \epsilon$ and $d(g^m \cdot x_0, h^n \cdot y_0) \leq \epsilon$ for some $m,n \in \Z - 0$, such that $d(x_0, g^m\cdot x_0) > K$. Now, $d(x_0, g^m \cdot x_0) \leq m \cdot d(x_0, g\cdot x_0) = m\cdot \lambda_X(g)$ and thus $m>\frac K{\lambda_X(g)}$.

Since $\pi$ is 1-Lipschitz, we have that $d(\pi(x_0), \pi(y_0)) \leq \epsilon$ and $d(g^m \cdot \pi(x_0), h^n \cdot \pi(y_0)) \leq \epsilon$. It follows that 
\[
\mathrm{diam}(\alpha_g \cap \alpha_h) \,\geq\, \lambda_T(g) \cdot m - 2\epsilon 
\,>\, \lambda_T(g)\frac K{\lambda_X(g)}-2\epsilon \,=\, \kappa + \lambda_T(g)\lambda_T(h).
\]
This contradicts \cref{lemma:tree-axes-intersections}.
\end{proof}

Suppose now that $h$ acts elliptically on $T$ and let $y_0 \in \pi^{-1}(\mathrm{Fix}_T(h))$. Let $A(h) = \{h^n \cdot y_0\}_{n\in \mathbb{Z}}$. The following is analogue of \cref{lemma:general-axes-intersections} in this situation.

\begin{lemma}\label{lemma:loxodromic-elliptic-axes-intersction}
For each $\epsilon > 0$, there exists a constant $K$
that depends only on $\epsilon$ and the conjugacy class of $g$ such that 
\[ \mathrm{diam}_X(N_{\epsilon}(A(g)) \cap A(h)) \leq K.\]
 \end{lemma}

 \begin{proof}
Let $K = \lambda_X(g) \frac{2\epsilon}{\lambda_T(g)}$ and suppose that $\mathrm{diam}(N_{\epsilon}(A(g)) \cap A(h)) > K$. 
    
As before, after replacing $y_0$ with $h^k \cdot y_0$ and $x_0$ with $g^l \cdot x_0$ for some $k,l \in \Z$, we may assume that $d(x_0,y_0 )\leq \epsilon$ and $d(g^m \cdot x_0, h^n \cdot y_0) \leq \epsilon$ for some $m,n \in \Z - 0$, such that $d(x_0, g^m\cdot x_0) > K$. Thus, since $\pi$ is 1-Lipschitz, we have that $d_T(\pi(x_0), \pi(y_0)) \leq \epsilon$ and $d_T(g^m \cdot \pi(x_0), h^n \cdot \pi(y_0))  = d_T(g^m \cdot \pi(x_0), \pi(y_0)) \leq \epsilon$.
Hence, $d_T(\pi(x_0) , g^m \cdot \pi(x_0)) \leq 2\epsilon$. Now since $\pi(x_0) \in \mathrm{Axis}_T(g)$, we have that $d_T(\pi(x_0) , g^m \cdot \pi(x_0)) = m \cdot\lambda_T(g)$, and thus $m \leq \frac{2\epsilon}{\lambda_T(g)}$. 

On the other hand, $d(x_0, g^m \cdot x_0) \leq m \cdot d(x_0, g\cdot x_0) = m\cdot \lambda_X(g)$ and thus
\[ m \,>\, \frac K{\lambda_X(g)} \,=\, \frac{2\epsilon}{\lambda_T(g)}.\]
This is a contradiction. 
\end{proof}

\section{The UPG case}

For the remainder of this article, a \emph{product subgroup} $P \leq G$ is assumed to be a non-trivial product, i.e. of the form $P \cong H \times K$ where $H$ and $K$ are both non-trivial.

The aim of this section is to prove the following:

\begin{proposition}\label{prop:main}
Let $G$ be a \{f.g.\ free\}-by-cyclic group with a UPG monodromy. There exists a quasitree $X(G)$ with an acylindrical $G$-action and a $G$-equivariant quasiisometry
\[ X(G) \to \mathrm{Cay}(G; S \cup \mathcal{P}),\]
where $\mathcal{P}$ is a set of orbit representatives of maximal product subgroups of $G$ and $S \subset G$ is a finite generating set. If the degree of polynomial growth is positive, then the action of $G$ on $X(G)$ is non-elementary.
\end{proposition}

\subsection{The construction} \label{sh:XG}

Let $G$ be infinite cyclic or \{f.g.\ free\}-by-cyclic with UPG monodromy. We define $X(G)$ by induction on the degree of growth of the monodromy of $G$. As part of the induction, we will show that there exists a $G$-equivariant 1-Lipschitz map 
\[ \pi \colon X(G) \to T_G\]
where $T_G$ is a $G$-tree. In the case that $G$ is not cyclic and not a product, $T$ is the Bass--Serre tree of the splittings described in \cref{sec:splittings}, and otherwise $T_G = X(G)$ and $\pi$ is the identity map.

\medskip

Suppose first that $G \cong \Z$. Then, we identify $X(G)$ with the real line $\R$ equipped with its standard cell structure with vertices at integer points. Taking the origin as the basepoint, the action of $G$ on $\R$ is by translation. If $G$ is \{f.g.\ free\}-by-cyclic where the monodromy grows polynomially of degree zero, then we set $X(G)$ to be a point and the action to be the trivial action.

\medskip

Suppose that the degree of growth is linear. Let $G \cong \pi_1(\mathbb{G}, x_0)$ be the splitting from \cref{prop:linear-splitting}. We set $X(G)$ to be the associated Bass--Serre tree and define the $G$-action via the isomorphism $G \cong \pi_1(\mathbb{G}, x_0)$. We set $T_G := X(G)$ and $\pi$ to be the identity map. %Recall from Proposition~\ref{prop:linear-splitting} that the induced action of $G$ on $X(G)$ is 4-acylindrical.

\medskip

Now let $G = F \rtimes_{\phi} \Z$ where $\phi$ represents a unipotent outer automorphism that grows polynomially of degree $d > 1$. Suppose that $X(G)$, $T_G$, and $\pi$ are defined for all \{f.g.\ free\}-by-cyclic groups with UPG monodromy of degree less than $d$. 

Let $G \cong \pi_1(\mathbb{G})$ be the topmost edge splitting of $G$ and let $T$ be the associated Bass--Serre tree. For each vertex $v \in V(T)$ and edge $e \in E(T)$, let $G_v$ be the stabiliser of $v$ and $G_e$ be the stabiliser of $e$, respectively, under the induced action of $G$ on $T$. For each $v \in V(T)$, let $X_v$ be a copy of $X(G_v)$ and for each $e \in E(T)$, let $\sigma_e$ denote the generator of $G_e$. %We define maps 
%\[\partial^{\pm}_e \colon X_e \to X_{e^{\pm}},\]
%as follows.

Fix $e \in E(T)$ and let $v \in\{e^+,e^-\}$. Let $T_{v} := T_{G_v}$ be the $G_v$-quasitree constructed by induction and $\pi_{v} \colon X_v \to T_v$ the associated 1-Lipschitz map. Let $g_v = \partial_{e}^{\pm}(\sigma_e) \in G_v$. Pick a vertex $y_0$ in the preimage of the minset of $g_v$ in $T_v$ that minimises $\dist_{X_v}(g_v\cdot y_0, y_0)$. We set $A(g_v) := \{g_v^n\cdot y_0\}_{n\in \mathbb{Z}}$. 

\begin{remark}
    Note that if an element $g \in G$ acts elliptically on $X(G)$ then it acts elliptically on $T$. Thus, there exists some vertex $v \in V(T)$ such that $g \in G_v$ and hence $g$ preserves $X_v$. By induction, we observe that if $g$ acts elliptically on $X(G)$ then it necessarily fixes a vertex of $X(G)$. Hence, each set $A(g_v)$ either consists of a single vertex $\{y_0\}$, or it is a quasigeodesic in $X_v$. 
\end{remark}

%If $g_{v}$ acts elliptically on $T_{v}$, pick a vertex in $T_v$ that is fixed by $g_v$ and let $y_0$ be a vertex in its preimage under the map $\pi_v \colon X_v \to T_v$. Set $A(g_v) := \{g_v^n \cdot y_0\}_{n\in \mathbb{Z}}$. If $g_v$ acts loxodromically on $T_v$, consider the axis of $g_v$ in $T_v$ and pick a vertex $y_0$ in the preimage of the axis that minimises $\dist_{X_v}(g_v\cdot y_0, y_0)$. Set $A(g_v) := \{g_v^n\cdot y_0\}_{n\in \mathbb{Z}}$. 

Now, for every edge $e \in E(T)$, if $A(g_{e^{+}}) = \{g_{e^+}^n\cdot y_0\}$ and $A(g_{e^{-}}) = \{g_{e^{-}}^n\cdot y_0'\}$, add an edge joining $g_{e^+}^n\cdot y_0$ to $g_{e^-}^n\cdot y_0'$ for every $n \in \mathbb{Z}$. Let $X(G)$ be the resulting space, equipped with the path metric. We will identify each vertex space $X_v$ with its image in $X(G)$, although we note that the map $X_v\hookrightarrow X(G)$ may not be a quasiisometric embedding. We let $\pi \colon X(G) \to T$ be the map that sends each point in $X_v$ to the vertex of $T$ labelled by $v$, and maps each edge joining a vertex in $X_v$ to a vertex in $X_w$, for $v\neq w$, over the edge $e = (v,w)$ in $T$. Clearly $\pi$ is 1-Lipschitz. Moreover, the action of each $G_v$ on $X_v$ extends uniquely to an action of $G$ on $X(G)$ so that the map $\pi \colon X(G) \to T$ is equivariant.

\begin{lemma} \label{lem:cayley}
$X(G)$ is equivariantly quasiisometric to $\Cay(G,S\cup\cal P)$ for some finite set $S\subset G$ and a collection $\mathcal{P}$ of maximal product subgroups of $G$ such that every maximal product subgroup of $G$ is conjugate to an element of $\mathcal{P}$.
\end{lemma}

\begin{proof}

If $G$ is cyclic or a product then the result is clear. 

Suppose now that $G$ has UPG monodromy of degree $d > 0$. Let $\operatorname{Lin}_{\leq}(G)$ denote the set of at most linearly growing \{f.g.\ free\}-by-cyclic subgroups appearing as vertex stabilisers in the cyclic hierarchy arising from the topmost edge splitting of $G$ (if $G$ has linearly growing monodromy then we take the cyclic hierarchy to be trivial and set $\operatorname{Lin}_{\leq}(G) = \{G\}$). 

Note that the vertex set $V(X(G))$ of the graph $X(G)$ is precisely $\bigsqcup_{H\in\operatorname{Lin}_{\leq}(G)}V(X(H))$. Recall that for $H \in \operatorname{Lin}_{\leq}(G)$, if $H \cong \Z$ then $X(H)$ is a line with a non-trivial $H$-action, if $H \cong F' \times \Z$ where $F' \neq 1$ then $X(H)$ is a point, and if $H$ is linearly growing then $X(H)$ is the Bass--Serre tree $T_H$ of $H$ for the splitting in \cref{prop:linear-splitting}. 

Let $A\subset V(X(G))$ be a finite set of orbit representatives. For each pair of subsets $B,C\subset A$, choose an element $g_{B,C}\in G$ with the property that $(g_{B,C}B)\cap A=C$ whenever such an element exists. Let $S$ be the finite set of elements $g_{B,C}$ chosen this way. By the standard argument from the proof of the Milnor--Schwarz Lemma \cite{milnor:curvature, schwarz:volume}, $X(G)$ is equivariantly quasiisometric to the Cayley graph of $G$ with respect to the generating set consisting of $S$ together with the stabilisers of the elements in $A$, which we denote by $\mathcal{P}$. 

Now, if $P \in \mathcal{P}$ stabilises a vertex of the topmost edge splitting of $G$ then it is a maximal product subgroup of $G$. Otherwise, it stabilises a vertex of the Bass--Serre tree $T_H$ of a linearly growing element $H \in \operatorname{Lin}_{\leq}(G)$, in which case it is either maximal or it is a $\mathbb{Z}^2$-subgroup that is contained in the stabiliser of every adjacent vertex to $v$ in $T_H$. In the latter case, there must exist some adjacent vertex $w$ to $v$ that is also an element of $A$. Hence, the stabiliser $P' := \mathrm{stab}_G(w)$ is an element of $\mathcal{P}$ and $P \leq P'$. Then, we may remove $P$ from $\mathcal{P}$ without altering the quasiisometry type of $\mathrm{Cay}(G; S \cup \mathcal{P})$. Continuing in this way, we can ensure that every element of $\mathcal{P}$ is a maximal product subgroup.

Conversely, by acylindricity of the cyclic hierarchy, any maximal product subgroup is contained in some unique non-cyclic element of $\operatorname{Lin}_{\leq}(G)$, and if it is contained in a linearly growing $H \in \operatorname{Lin}_{\leq}(G)$, then by acylindricity of the splitting in \cref{prop:linear-splitting}, it must be elliptic for the action on $T_L$, and thus conjugate to some element of $\mathcal{P}$.
\end{proof}

%%%%%%%%%%%%%%%%%%%%%%%%%%%%%%
\subsection{Quasitree and acylindricity}

The goal of this subsection is to show that $X(G)$ is a quasitree and the $G$-action is acylindrical. Throughout the section, we will fix $T = T_G$ to be the tree and $\pi \colon X(G) \to T$ the 1-Lipschitz map constructed in \cref{sh:XG}. 
We begin with the following lemma:

\begin{lemma} \label{lem:bounded_intersections}
If $e_1$ and $e_2$ are edges of $T$ with $e_1^+=e_2^\pm$, then the images of the edge-inclusions of $X_{e_1}$ and $X_{e_2}$ in $X_{e_1^+}$ have uniformly bounded coarse intersection.
\end{lemma}

\begin{proof}
Let $v = e_1^{+} = e_2^{\pm}$. By the proof of \cite[Lemma~5.2]{kudlinskavaliunas:free}, the pointwise stabiliser of a path of the form $e_1 e_2$ or $e_1 \bar{e}_2$ in $T$ is trivial.  Thus, it must be the case that the images $\partial_{e_1}^{+}(\sigma_{e_1})$ and $\partial_{e_2}^{\pm}(\sigma_{e_2})$ are distinct.

Since the action of $G$ on $T$ is cocompact, there are finitely many conjugacy classes of elements of the form  $\partial_{e_1}^{+}(\sigma_{e_1})$ and $\partial_{e_2}^{\pm}(\sigma_{e_2})$ over all edges $e_1$ and $e_2$ in $T$ with $e_1^{+} = e_2^{\pm}$.

As discussed in \cref{sec:splittings}, the element $\partial_{e_1}^{+}(\sigma_e)$ is loxodromic for the action of $G_v$ on the Bass--Serre tree $T_v$. Hence, by \cref{lemma:general-axes-intersections} and \cref{lemma:loxodromic-elliptic-axes-intersction}, either $\partial_{e_1}^{+}(\sigma_e)$ and $\partial^\pm_{e_2}(\sigma_{e_2})$ have a common power, or their axes in $X_v$ have uniformly bounded coarse intersection. 

By \cref{lemma:common-powers}, if $\partial_{e_1}^{+}(\sigma_e)$ and $\partial^\pm_{e_2}(\sigma_{e_2})$ have a common power then they must be equal. The result now follows. \end{proof}

%As described in \cref{sec:splittings}, the splitting of $G$ associated to the action on $T$ arises from a unipotent representative $(f, \Gamma)$ of the monodromy. In particular,  distinct edges with nontrivial suffixes have distinct suffixes. If $e$ is an edge of $\Gamma$ with $e\mapsto e \rho$, then the corresponding maps $\partial_e^\pm$ are $\partial_e^+(\sigma_e)=t_{e^+}g_{\rho}^{-1}$ and $\partial_e^-(\sigma_e)=t_{e^-}$, where $t_{e^\pm}$ denotes the stable letter of the vertex group $G_{e^\pm}$ and $g_{\rho}$ denotes the element of $G_{e^+}$ represented by $\rho$.

% If either of the two elements acts elliptically on $X_v$ then there is nothing to prove. Otherwise since the action of $G_v$ on $X_v$ is acylindrical, there is a constant $C$ such that if the axes of the two elements have coarse intersection of diameter greater than $C$, then their axes lie at finite Hausdorff-distance. It follows from \cite[Lem.~6.5]{dahmaniguirardelosin:hyperbolically} and \cite[Lem.~6.8]{osin:acylindrically} that the two elements have a common power (up to inverting one of them). But this is impossible by \cref{lemma:common-powers} and the assumption that nontrivial suffixes are distinct.

Given $x\in X$ and a vertex $v\in T$, let $\pi_v(x)$ be a closest point in $X_v\subset X$ to $x$.

\begin{lemma} \label{lem:shrinking}
If $\dist_T(v,w)>2$, then $\diam\pi_v(X_w)$ is uniformly bounded. %If $v,w\in T$ are distinct and $\pi_vX_w$ is unbounded, then $\pi_vX_w$ is equal to the image of an edge inclusion.
\end{lemma}

\begin{proof}
Suppose $\dist_T(v,w)=n>2$, and let $(e_1,e_2,\dots,e_n)$ be an edge path from $v$ to $w$. As $n>2$, there is some $m<n$ such that $e_m$ is oriented towards $w$. By Lemma~\ref{lem:bounded_intersections}, the images of the edge-inclusions of $X_{e_m}$ and $X_{e_{m+1}}$ into $X_{e_m^+}$ have uniformly bounded coarse intersection. Since the image of $X_{e_m}$ separates $X_v$ from $X_w$, this proves the lemma. 
\end{proof}

We are now ready to show that $X(G)$ is a quasitree. As before, we argue by induction on the degree of growth. 

Given $v\in T$, let $N(v)$ be the set of neighbours of $v$ in $T$ and set 
\[
\bar N(v) \,=\, N(v)\cup\{w\in T\,:\,\diam\pi_wX_v=\infty\}.
\]
Note that $v\in\bar N (v)$, and if $w\in\bar N(v)$ then $v\in\bar N(w)$. If $w\in\bar N(v)$, then let $\bar\pi_wX_v$ be a uniform thickening of $\pi_wX_v$ with the property that each connected component of $X_w\ssm\pi_wX_v$ has uniformly bounded coarse intersection with $\bar\pi_wX_v$. Such a neighbourhood exists because $X_w$ is a quasitree. 

Define $\bar X_v=\bigcup_{w\in\bar N(v)}\bar\pi_wX_v$. Intuitively, $\bar X_v$ is just $X_v$ with fins and spines glued to it: most fins and all spines have length one, but some fins can have length two if the vertex space in the middle corresponds to a vertex of $\Gamma$ that is not the terminus of any edge. Every point of $\bar X_v$ has a unique closest point in $X_v$, from which it is at a uniform distance, by Lemma~\ref{lem:shrinking}.

\begin{proposition} \label{prop:quasitree}
$X(G)$ is a quasitree.
\end{proposition}

\begin{proof}
We use Manning's bottleneck criterion \cite{manning:geometry} in the following form (\emph{cf.} \cite[\S3.6]{bestvinabrombergfujiwara:constructing}): \emph{a geodesic metric space $Y$ is a quasitree if and only if there is a constant $\Delta$ such that for each pair $x,y\in Y$ there is a path $p$ from $x$ to $y$ that lies in the $\Delta$--neighbourhood of every path from $x$ to $y$.}

\medskip

First suppose that there is some $v\in T$ such that $x,y\in\bar X_v$. In this case, let $p_{xy}$ be the concatenation of: a geodesic from $x$ to $\pi_v(x)$; an $X_v$-geodesic from $\pi_v(x)$ to $\pi_v(y)$; and a geodesic from $\pi_v(y)$ to $y$. 

Suppose $\alpha$ is a path in $X(G)$ from $x$ to $y$. We can decompose $\alpha$ as $\alpha=\alpha_0\beta_1\alpha_1\dots\alpha_n$, where $\alpha_i$ is a maximal subpath lying in $\bar X_v$. Let $\eta_i$ be a geodesic in $\bar X_v$ joining the terminal point of $\alpha_{i-1}$ to the initial point of $\alpha_i$, and let $\alpha'=\alpha_0\eta_1\alpha_1\dots\alpha_n$. It follows from Lemma~\ref{lem:bounded_intersections} that the endpoints of each $\beta_i$ are at uniformly bounded distance, and hence $\alpha'$ is contained in a uniform neighbourhood of $\bigcup_{i=0}^n\alpha_i$.

Since $\bar X_v$ is contained in a uniform neighbourhood of $X_v$, the path $\pi_v\alpha'$ is contained in a uniform neighbourhood of $\alpha'$, hence in a uniform neighbourhood of $\alpha$. But now $\pi_v\alpha'$ is a path in the quasitree $X_v$ from $\pi_v(x)$ to $\pi_v(y)$, so it has a uniform neighbourhood containing $p_{xy}$, and hence so does $\alpha$.

\medskip

Now consider a general pair of points $x,y\in X$. Let $v,w\in T$ be such that $x\in X_v$, $y\in X_w$, and let $(v=v_0,v_1,\dots,v_n=w)$ be the geodesic in $T$ from $v$ to $w$. We wish to choose a path $p_{xy}$ from $x$ to $y$. The idea is illustrated in Figure~\ref{fig:pxy}. Let $i_0=0$. Given $i_j$, let $i_j^+>i_j$ be maximal such that $v_{i_j^+}\in\bar N(v_{i_j})$, and if $i_{j^+}\ne n$ then set $i_{j+1}$ to be maximal such that $v_{i_j^+}\in\bar N(v_{i_{j+1}})$. This gives a subsequence $(i_0,\dots,i_m)$ of $(0,\dots,n)$. 

By construction, Lemma~\ref{lem:bounded_intersections} implies that the sets $\bar X_{v_{i_j}}$ have uniformly bounded coarse intersection. For each $j$, let $x_j$ be a closest-point in $\bar X_{v_{i_j}}$ to $x$, which is coarsely well defined, and similarly for $y$. Note that since $i_m^+=n$, we have that $y_m\in X_w$. For notational convenience, we write $x_0=x$ and $x_{m+1}=y$.

\begin{figure}[ht]
\begin{center} \makebox[0pt]{\begin{minipage}{1.3\textwidth}
\centering\includegraphics[width=\linewidth, trim= 5mm 5mm 0 0]{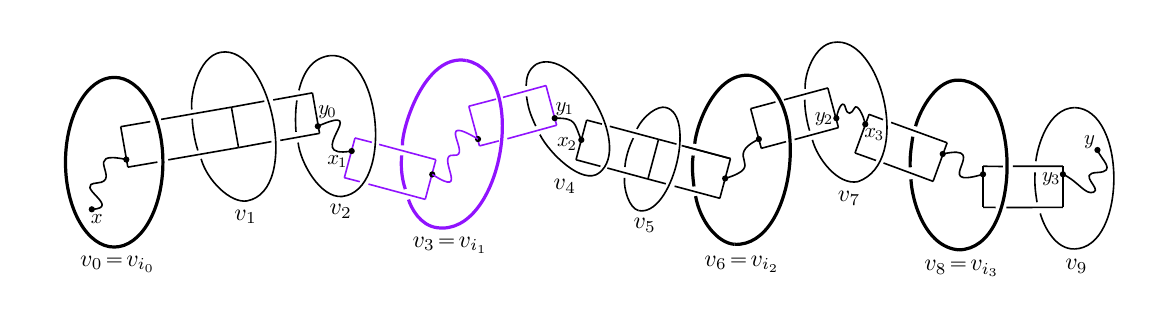}
\end{minipage}}\end{center}
\begin{center} \makebox[0pt]{\begin{minipage}{1.1\textwidth}
\caption{The construction of the path $p_{xy}$ in the proof of Proposition~\ref{prop:quasitree}. The sequence of $X_{v_{i_j}}$ is in bold. Starting with $i_0=0$, we have $i_0^+=2$ because $v_2\in\bar N(v_0)$, and then $i_1=3$. Next $i_1^+=4$, and $i_2=6$ because $v_4\in\bar N(v_6)$. The subspace $\bar X_{v_{i_1}}$ is highlighted in purple. The points $x_1$ and $y_1$ are produced by projecting to it.} \label{fig:pxy}
\end{minipage}}\end{center}
\end{figure}

Since $x_j,y_j\in\bar X_{v_{i_j}}$, we have already defined a path $p_{x_jy_j}$ between them. Similarly, we have already defined $p_{y_j,x_{j+1}}$ to be a geodesic $[y_j,x_{j+1}]$ in $X_{v_{i_j^+}}$, and since the sets $\bar X_{v_{i_j}}$ have uniformly bounded coarse intersection, the definition of $y_j$ and $x_{j+1}$ implies that $[y_j,x_{j+1}]$ is at uniform Hausdorff-distance from every shortest geodesic joining $\pi_{v_{i_j^+}}X_{v_{i_j}}$ and $\pi_{v_{i_j^+}}X_{v_{i_{j+1}}}$. We therefore define $p_{xy}$ to be the concatenation
\[
p_{xy} \,=\, p_{xy_0}[y_0,x_1]p_{x_1y_1}[y_1,x_2]\dots p_{x_my_m}[y_m,y].
\]
Our goal is to show that every path from $x$ to $y$ has a uniform neighbourhood containing $p_{xy}$.

Let $\alpha\subset X$ be a path from $x$ to $y$. Since $T$ is a tree, $\pi\alpha\subset T$ contains the geodesic from $v$ to $w$. In particular, for each $j$ there must be a subpath $\alpha_j$ of $\alpha$ that starts in $\pi_{v_{i_j}}X_{v_{i_j-1}}$ and ends in $\pi_{v_{i_j}}X_{v_{i_j+1}}$. Since $X_{v_{i_j}}$ is a quasitree, a uniform neighbourhood of $\alpha_j$ contains a geodesic in $X_{v_{i_j}}$ between its endpoints. From the property of $y_j$ and $x_{j+1}$ observed above, this implies that $[y_j,x_{j+1}]$, and in particular its endpoints, lie in a uniform neighbourhood of $\alpha_j$. 

We have shown that there is a uniformly bounded perturbation $\alpha'$ of $\alpha$ that: is a path from $x$ to $y$; passes through every $x_j$ and every $y_j$; and has a uniform neighbourhood that contains every $[y_j,x_{j+1}]$. Since $x_j,y_j\in\bar X_{v_{i_j}}$, our above arguments for points in a common vertex space show that every $p_{x_jy_j}$ lies in a uniform neighbourhood of $\alpha'$ as well. This shows that $p_{xy}$ lies in a uniform neighbourhood of $\alpha$, as desired.
\end{proof}

Recall that given $x\in X$ and a vertex $v\in T$, we let $\pi_v(x)$ be a closest point in $X_v\subset X$ to $x$. Let $L\ge1$ be such that every map $\pi_v$ is $L$--coarsely Lipschitz. Let $N\ge1$ be such that if $\dist_T(v,w)\ge N$ then $\diam\pi_v(X_w)<N$. Both $L$ and $N$ can be bounded in terms of the quasitree constant of $X$. 

By induction, the action of each vertex group $G_v$ on its associated quasitree $X_v$ is acylindrical. Given $\eps$, let $r(\eps)\ge1$ and $n(\eps)$ be such that in every vertex group $G_v$, the $\eps$--coarse stabilisers of pairs in $X_v$ at distance at least $r(\eps)$ has cardinality at most $n(\eps)$.

\begin{lemma} \label{lem:acylindrical}
The action of $G$ on $X(G)$ is acylindrical.
\end{lemma}

\begin{proof}
Given $\eps\ge1$, let $R=6NL\eps r(\eps)$. Suppose that $x,y\in X$ have $\dist(x,y)>R^2$. If $\dist_T(\pi(x),\pi(y))>R$, then any $g\in G$ that moves both $x$ and $y$ by at most $\eps$ must stabilise a geodesic in $T$ of length $R-2\eps$. This is only possible if $g=1$, because the action of $G$ on $T$ is 4--acylindrical, as described in \cref{sec:splittings}.

Otherwise, $\dist_T(\pi(x),\pi(y))\le R$. Since $\dist(x,y)>R^2$, there must be some vertex $v\in[\pi(x),\pi(y)]_T$ for which $\dist(\pi_v(x),\pi_v(y))>R$. Suppose that $g\in G$ moves both $x$ and $y$ by at most $\eps$. If $g$ is loxodromic on $T$, then $\dist_T(v,g^Nv)\ge N$. However, $\dist(x,g^Nx)\le N\eps$ and $\dist(y,g^Ny)\le N\eps$, which implies that 
\[
\dist(\pi_v(g^Nx),\pi_v(g^Ny)) \,\ge\, \dist(\pi_v(x),\pi_v(y)) - 2NL\eps-2L \,\ge\, R-2NL\eps-2L \,\ge N,
\]
contradicting the fact that $\diam\pi_v(X_{g^Nv})<N$. 

Thus $g$ must act elliptically on $T$. If $v_0\in T$ is fixed by $g$, then $g[v_0,v]_T=[v_0,gv]_T$, so there exists a fixed-point $w$ of $g$ in the geodesic $[v,gv]_T$. That is, $g\in G_w$ for some $w\in[v,gv]_T$. We know that $\dist(\pi_v(x),\pi_v(y))>R$ and, via the triangle inequality, that $\dist(\pi_{gv}(x),\pi_{gv}(y))>R-2L\eps-2L$, so it follows that $\dist(\pi_w(x),\pi_w(y))>R-2L\eps-2L$ as well. But $g\in G_w$, so since $R-2L\eps-2L\ge r(\eps)$, there are at most $n(\eps)$ such elements of $G_w$. We need to bound how many vertices $w$ can arise in this way. But we know that $\dist(x,gx)<\eps$, and hence $\dist(\pi(x),\pi(gx))\le\eps$, and similarly for $y$. Since $v\in[\pi(x),\pi(y)]_T$, this implies that $\dist(v,gv)\le\eps$. In particular, there are at most $\eps+1$ possible such vertices $w$, so $g$ is confined to a set of cardinality at most $(\eps+1)n(\eps)$.
\end{proof}

\section{The general case} \label{sec:general}

In this section, we use \cref{prop:main} to produce a largest acylindrical action for general finitely generated free-by-cyclic groups. Note that the actions in the following are not required to be cobounded.

\begin{definition}\label{defn:accessibility}
A group $G$ is \emph{strongly $\cal{AH}$-accessible} if it has an acylindrical action on a hyperbolic space $\calx$ such that for every other acylindrical action of $G$ on a hyperbolic space $\caly$, we have that $G \acts \caly \prec G \acts \calx$, i.e.\ there exist $x\in \calx$, $y\in \caly$, $r>0$ such that $\dist(y,gy)\le r\dist(x,gx)+r$ for all $g\in G$.
\end{definition}

By the Milnor--Schwarz Lemma \cite{milnor:curvature, schwarz:volume}, each equivalence class of the relation induced by $\prec$ is represented by some (possibly infinite) generating set of $G$. 

\begin{lemma} \label{lem:any_generating_set}
Let $S$ and $T$ be finite generating sets for a group $G$. For every subset $X\subset G$, we have that $G\acts\Cay(G;S\cup X)\asymp G\acts\Cay(G;T\cup X)$.
\end{lemma}

\begin{proof}
There exists $r$ such that every element of $S$ can be written as a word of length at most $r$ in $T$. Hence $G\acts\Cay(G;S\cup X)\preceq G\acts\Cay(G;T\cup X)$. The other direction is similar.
\end{proof}

% \begin{lemma}
%     Let $G$ be a \{f.g.\ free\}-by-cyclic group with UPG monodromy and let $\mathcal{P}$ be a collection of maximal product subgroups of $G$. Then, there exists a finite generating set $S \subseteq G$ such that for every acylindrical action of $G$ on a hyperbolic space $X$, we have that 
%     \[ G \acts X \preceq \mathrm{Cay} (G; S \cup \bigcup_{P \in \mathcal{P}} P).\] That is, $G$ is strongly accessible.\mk{Is it also true for any finite S?}
% \end{lemma}
We will use the following results to bootstrap $\AH$-accessibility for \{f.g.\ free\}-by-cyclic groups with UPG monodromy to all finitely generated free-by-cyclic groups.

\begin{lemma}[{\cite[Lemma 6]{minasyanosin:correction}}]\label{lem:finite-index-bootstrap}
Let $G$ be a group and $H \leq G$ a normal subgroup of finite index that is \AH-accessible. Let $X$ be a generating set of $H$ that realises \AH-accessibility of $H$. If $Y$ is a set of coset representatives of $H$ in $G$ that includes the identity, then the natural inclusion map
\[ \mathrm{Cay}(H, X) \to \mathrm{Cay}(G, X \cup Y)\]
is an $H$-equivariant quasi-isometry. Moreover, the natural action of $G$ on $\mathrm{Cay}(G, X \cup Y)$ is acylindrical.
\end{lemma}

\begin{theorem}[{\cite[Thm~7.9]{abbottbalasubramanyaosin:hyperbolic}}]\label{thm:strong-AH-rel-hyp}
Let $G$ be a group that is hyperbolic relative to a collection of strongly $\mathcal{AH}$-accessible subgroups $H_1, \ldots, H_n \leq G$. For each $1 \leq i \leq n$, let $Y_i$ be a generating set for $H_i$ whose equivalence class realises strong $\mathcal{AH}$-accessibility. There is a finite generating set $S$ of $G$ such that $[S \cup Y_1 \cup \ldots \cup Y_n]$ realises strong $\mathcal{AH}$-accessibility for $G$.
\end{theorem}

We can now prove the acylindricity part of \cref{thm:intro}. 

\begin{theorem}\label{thm:main}
Let $G$ be a free-by-cyclic group with a finite generating set $S$, and let $\cal P$ be a set of orbit representatives of maximal product subgroups. The Cayley graph $\mathrm{Cay} (G; S \cup \mathcal{P})$ is hyperbolic and the natural action of $G$ is acylindrical. %If $G$ is a \{f.g.\ free\}-by-cyclic group with $PG$ monodromy then $\mathrm{Cay} (G; S \cup \mathcal{P})$ is a quasitree.
    
Moreover, for every acylindrical action of $G$ on a hyperbolic space $\caly$, we have that 
\[ G \acts \caly \,\preceq\, \mathrm{Cay} (G; S \cup \mathcal{P}).\] 
In particular, $G$ is strongly \AH-accessible.
\end{theorem}

\begin{proof}
We begin by assuming that $G = F_n\rtimes_{\phi} \Z$ where $\phi$ is a representative of a UPG outer automorphism. By \cref{prop:main} $\mathrm{Cay}(G; S \cup \mathcal{P})$ is a quasitree and the action of $G$ on $\mathrm{Cay}(G; S \cup \mathcal{P})$ is acylindrical. If we add a cone point for each coset of an element of $\cal P$, then each nontrivial vertex stabiliser $\mathrm{stab}_G(v)$ is a maximal product subgroup of $G$. Hence, for every acylindrical action of $G$ on a hyperbolic space $\caly$, the induced action of $\mathrm{stab}_G(v)$ on $\caly$ has bounded orbits. It follows by \cite[Proposition~4.13]{abbottbalasubramanyaosin:hyperbolic} that
\[G \acts \caly \,\preceq\, \mathrm{Cay} (G; S \cup \mathcal{P}). \]

Suppose now that $G$ is \{f.g.\ free\}-by-cyclic with PG monodromy and let $G' \trianglelefteq_{f.i} G$ be a normal subgroup of finite index such that $G'$ admits UPG monodromy. By the previous argument, $G'$ has a largest acylindrical action, represented by $S' \cup \mathcal{P}$ where $S'$ is a finite generating set of $G'$ and $\mathcal{P}$ the set of maximal product subgroups. Let $T$ be a collection of coset representatives of $G'$ in $G$ that contains the identity. By \cref{lem:finite-index-bootstrap}, we have that $\mathrm{Cay}(G; S' \cup T \cup \mathcal{P})$ is a quasitree and the action of $G$ is acylindrical. Thus, setting $S = S' \cup T$, we obtain a finite generating set $S$ of $G$ such that $\mathrm{Cay}(G; S \cup \mathcal{P})$ is a quasitree and the action of $G$ is acylindrical. The same holds for an arbitrary finite generating set of $G$ by \cref{lem:any_generating_set}.
By the same argument as above, it follows that 
\[G \acts \caly \,\preceq\, \mathrm{Cay} (G; S' \cup \mathcal{P}) \]
for every acylindrical action of $G$ on a hyperbolic space $\caly$.
    
Finally, suppose that $G$ is finitely generated free-by-cyclic. By \cref{prop:rel-hyp}, there is a finite collection $\calh$ of \{f.g.\ free\}-by-cyclic subgroup of $G$ with PG monodromies such that $G$ is hyperbolic relative to $\calh$. Combining \cref{thm:strong-AH-rel-hyp} with the observations above yields the result.
\end{proof}

%%%%%%%%%%%%%%%%%%%%%%%%%%%%%%%%%%%%%%%%%%%%%%%%%%
\section{Morseness and stability} \label{sec:stability}

Let $X$ be a metric space and let $M \colon \mathbb{R}_{\geq 0} \times \mathbb{R}_{\geq 0} \to \mathbb{R}$ be a function. A quasigeodesic $\gamma$ in $X$ is said to be \emph{(M-)Morse} if every $(K,C)$-quasigeodesic $\eta$ with endpoints on $\gamma$ lies in the $M(K,C)$-neighbourhood of $\gamma$. We call $M$ a \emph{Morse gauge} for $\gamma$. 

% A map $\gamma \colon I \to X$ is a \emph{local Morse quasigeodesic} if there exists a constant $L \geq 0$ such that for every subinterval $J \subseteq I$ of diameter at most $L$, the restriction $\gamma|_J$ is a Morse quasigeodesic. A space $X$ is roughly said to be \emph{Morse local-to-global} if every local Morse quasigeodesic is a Morse quasigeodesic (see \cite[Definition B]{russellsprianotran:local} for the full definition).

% \begin{definition}[{\cite[Definition~4.17]{russellsprianotran:local}}]
%     A space $Y$ is \emph{Morse detectable} if there exists a $\delta$-hyperbolic space $X$ and a coarsely Lipschitz map $\pi \colon Y \to X$ such that for every $(K,C)$-quasigeodesic $\gamma \colon I \to Y$, we have the following:
%     \begin{enumerate}
%         \item if there exists a Morse gauge $M$ such that $\gamma$ is $M$-Morse, then $\pi \circ \gamma$ is a $(\lambda, \varepsilon)$-quasigeodesic, where $\lambda$ and $\varepsilon$ depend on $\delta, K, C$ and $M$;
%         \item if there exist $\lambda \geq 1$ and $\varepsilon \geq 0$ such that $\pi \circ \gamma$ is a $(\lambda, \varepsilon)$-quasigeodesic then $\gamma$ is $M$-Morse, where the Morse gauge $M$ depends on $\delta, K,C, \lambda$ and $\varepsilon$.
%     \end{enumerate}
% \end{definition}

\begin{proposition}\label{lemma:Morse-detectable}
Let $G$ be a free-by-cyclic group with a finite generating set $S$. Let $X = \mathrm{Cay}(G; S \cup \mathcal{P})$ be as in \cref{thm:main} and let $\hat\iota \colon \mathrm{Cay}(G,S) \to X$ be the natural embedding. Let $\gamma \colon I \to \mathrm{Cay}(G, S)$ be a $(K,C)$-quasigeodesic. The following are equivalent.
    \begin{enumerate}
        \item \label{it:1} $\gamma$ is Morse. 
        \item \label{it:2} $\gamma$ has uniformly bounded intersection with all cosets of elements in $\mathcal{P}$.
        \item \label{it:3} $\hat\iota \circ \gamma \colon I \to X$ is a quasigeodesic.
    \end{enumerate}
    Moreover, the Morse gauge of $\gamma$ can be chosen to only depend on $K,C,$ and the quasigeodesic constants of $\hat\iota \circ \gamma$, and conversely the quasigeodesic constant of $\hat\iota \circ \gamma$ can be chosen to only depend on $K,C$ and the Morse gauge of $\gamma$.
\end{proposition}

% \begin{lemma}\label{lemma:PG-Morse-detectable}
%     Let $G$ be a f.g.\ free-by-cyclic group and $S$ a finite generating set. Then, $\mathrm{Cay}(G;S)$ is Morse detectable.
% \end{lemma}

\begin{proof}
Let $X = \mathrm{Cay}(G; S \cup \mathcal{P})$ be as in \cref{thm:main}, and let $\hat\iota \colon \mathrm{Cay}(G; S) \to X$ be the natural 1-Lipschitz embedding. Firstly, note that if $\hat\iota \circ \gamma$ is a quasigeodesic, then \cite[Lemma~4.6]{charneycordessisto:complete} states that $\gamma$ is $M$-Morse, because $G$ acts acylindrically on $X$. Secondly, if $\gamma$ is Morse then it must have uniformly bounded intersection with the elements of $\cal P$, because they are product subgroups that are undistorted \cite[Lemma~4.1]{mutanguha:onpolynomial}.

\medskip

Now suppose that $\gamma\colon I\to\Cay(G;S)$ intersects each coset of each element of $\mathcal{P}$ in a subset of uniformly bounded diameter. It remains to show that $\hat\iota\circ\gamma$ is a uniform quasigeodesic. By \cref{lem:cayley}, there is no loss in replacing $X$ by the space constructed in Section~\ref{sh:XG}, so we shall conflate the two. 

If $G$ is \{f.g.\ free\}-by-cyclic with linearly-growing and unipotent monodromy, then $X$ is the Bass--Serre tree of the $\Z^2$-splitting from \cref{prop:linear-splitting}. Since $\gamma$ has uniformly bounded intersections with the vertex-stabilisers, the path $\hat\iota\circ\gamma$ must be a quasigeodesic in the tree $X$.

\medskip 

Next suppose that the monodromy of $G$ is unipotent and polynomially growing of degree $d>1$, and assume that we know the result for all degrees less than $d$. From the construction of $X$, one could appeal to the description of quasigeodesics in \emph{trees of hyperbolic spaces} given in \cite[\S7]{kapovichsardar:trees}, but we shall prove that $\hat\iota\circ\gamma$ is a quasigeodesic directly. 

Let $T$ be the Bass--Serre tree of the topmost edge splitting of $G$, and let $\pi:G\to T$ be the associated projection. There is no loss in replacing $\Cay(G;S)$ with the tree of spaces obtained from this splitting by setting the vertex spaces to be Cayley graphs of the vertex groups with respect to finite generating sets, so again we shall conflate the two. Thus $G$ has vertex spaces $\{G_v\,:\,v\in T\}$ and two adjacent vertex spaces are joined by a strip $\Sigma_e$ between $\partial_e^-(\sgen{\sigma_e})$ and $\partial_e^+(\sgen{\sigma_e})$. Write $\Sigma_e^\pm\subset G_{e^\pm}$ for the two sides of the strip. There are finitely many isometry types of such strips.

By induction, we know that if $\beta\subset\gamma$ is a connected subpath such that $\beta\subset G_v$ for some vertex $v\in T$, then $\hat\iota\circ\beta$ is a uniform quasigeodesic in the quasitree $X_v$ with the intrinsic metric, though in general $X_v$ will be distorted in $X$.

\begin{claim} \label{claim:stay_quasi}
$\hat\iota\circ\beta$ is a uniform quasigeodesic of $X$.
\end{claim}

\begin{claimproof}
First of all, we remark that $\hat\iota\circ\beta$ is a uniformly Lipschitz path in $X$, because $\hat\iota$ is 1--Lipschitz. We must show that it is also uniformly colipschitz. 

Recall that $X$ is built inductively by taking vertex spaces $X_w=X(G_w)$ and adding edges from $A(g_{e^+})\subset X_{e^+}$ to $A(g_{e^-})\subset X_{e^-}$ for every edge $e$ of $T$. The metric space $(X_v,\dist_X)$ can therefore be thought of as being obtained from the quasitree $(X_v,\dist_{X_v})$ by coning off axes of certain loxodromic elements. If $e^\pm=v$, then the axis $A(g_v)\subset X_v$ gets coned off only if $g_{e^\pm}$ fixes a vertex of $X_{e^\mp}$, in which case $g_{e^\mp}$ is contained in a product subgroup of $G_{e^\mp}$, and hence of $G$. 

Since $\beta$ intersects each coset of each element of $\cal P$ in a subset of uniformly bounded diameter, it follows that $\hat\iota\circ\beta$ spends a uniformly bounded amount of time in the axes of $X_v$ that get coned off. Thus $(\hat\iota\circ\beta,\dist_X)$ is colipschitz with constant a uniform amount worse than that of $(\hat\iota\circ\gamma,\dist_{X_v})$. 
\end{claimproof}

% \hp{The following could be added to the above paragraph. It adds extra objects and maps to the argument, but maybe it's worth it? What do you think? \\
% Fix a uniform rough isometry from $X_v$ to a tree $\cal T$ and cone off things there. Push $\hat\iota\gamma$ along to $\cal T$. In the cone-off of $\cal T$, its constants decay a controlled amount. When you pull back you get a uniform quasigeodesic in $X$.}

By considering the components of the intersections of $\gamma$ with the vertex spaces $G_v$, we now know that $\hat\iota\circ\gamma$ is a concatenation of uniform quasigeodesics, but this is not enough to conclude that it is a quasigeodesic itself. First we handle the overlaps between adjacent vertex spaces of $X$.

Given a vertex $v$ of $T$, recall the thickening $\bar X_v$ defined before \cref{prop:quasitree}. Let $J\subset I$ be a maximal subinterval such that $\hat\iota\circ\gamma|_J\subset\bar X_v$. The complement $\bar X_v\ssm X_v$ is a disjoint union of sets of diameter two and strips, all of which correspond to quasigeodesic edge-strips (possibly of width two) in the tree of spaces $G$. Because of this, the fact that $\gamma$ is a quasigeodesic in $G$ together with Claim~\ref{claim:stay_quasi} implies that $\hat\iota\circ\gamma|_J$ is a uniform quasigeodesic in $X$. 

We now show that $\pi\circ\gamma$ does not make meaningful backtracks in $T$.

\begin{claim} \label{claim:backtrack}
If $x=\gamma(s)$ and $y=\gamma(t)$, with $s<t$, are such that $\hat\iota(x),\hat\iota(y)\in\bar X_v$ but $\hat\iota(\gamma(r))\notin \bar X_v$ for all $r\in(s,t)$, then $\dist_G(x,y)$ is uniformly bounded.
\end{claim}

\begin{claimproof}
As $T$ is a tree, there is a vertex $w\in\bar N(v)$ such that $\pi(x)=\pi(y)=w$. By the construction of $\bar X_v$, each connected component of $X_w\ssm\bar X_v$ has uniformly bounded coarse intersection with $\pi_w(X_v)\subset\bar X_v$ in the metric space $(X_w,\dist_{X_w})$. In particular, since $\hat\iota\circ\gamma(s,t)$ is disjoint from $\bar X_v$ but its endpoints $x$ and $y$ have $\hat\iota(x),\hat\iota(y)\in\bar X_v$, it follows that $\dist_{X_w}(\hat\iota(x),\hat\iota(y))$ is uniformly bounded. 
% N.B. the thickening to make small intersections only adds vertices that are separated by edge-strip edges, so not making distance in $X_v$ means they don't make distance in $G$.

Let $e$ be the edge of $T$ that contains $w$ and is part of the geodesic from $v$ to $w$. Note that either $e=vw$ or $\dist_T(v,w)=2$. Let $u$ denote the vertex of $e$ that is not $w$. For concreteness, let us assume that $u=e^-$, though the argument is the same in either case. We have $x,y\in\Sigma_e^+\subset G_w$. 

When we build $X$, we attach edges to the set $\pi_w(X_v)\subset X_w$ that connect it to the vertex space $X_u$. The set $\pi_w(X_v)$ is either a vertex or it is an axis of a loxodromic isometry of $(X_w,\dist_{X_w})$, namely $g=\partial_e^+(\sigma_e)$. 

If $\pi_w(X_v)$ is a vertex of $X_v$, then $x$ and $y$ lie in the stabiliser of that vertex, which is a product subgroup of $G$. Since $\{x,y\}\subset\gamma$, the distance in $G$ from $x$ to $y$ is uniformly bounded by the assumption on $\gamma$. 

Otherwise, $\pi_w(X_v) =  \{g^n \cdot y_0\}_{n\in \Z}$ is an axis of the loxodromic isometry $g$ of $X_w$. In particular, $g$ is not contained in any product subgroup of $G_w$. Since there are only finitely many isometry types of strips $\Sigma_e$ in $G$, there is a uniform bound on the diameter of the intersections of the cosets of $\langle \partial_e^+(\sigma_e) \rangle$ with cosets of the product subgroups in $G_w$. By induction on the degree of polynomial growth, we deduce that $\hat\iota\circ\Sigma_e^+$ is a uniform quasigeodesic of $(X_w,\dist_{X_w})$. It therefore restricts to a uniform quasigeodesic of $(X_w,\dist_{X_w})$ from $\hat\iota(x)$ to $\hat\iota(y)$. Since $\dist_{X_w}(\hat\iota(x),\hat\iota(y))$ is uniformly bounded, this gives a uniform bound on $\dist_{G_w}(x,y)$. As $G_w$ is undistorted in $G$ \cite[Lemma~4.1]{mutanguha:onpolynomial}, this establishes the claim.\end{claimproof}

We can now prove that $\hat\iota\circ\gamma$ is a uniform quasigeodesic of $X$ in the UPG case. Indeed, by Claim~\ref{claim:backtrack}, there is a uniformly bounded perturbation $\gamma'$ of $\gamma$ such that $\pi\circ\gamma'$ does not backtrack in $T$. For each vertex $v$ appearing in $\pi\circ\gamma$, we know from the discussion after Claim~\ref{claim:stay_quasi} that $\hat\iota\circ\gamma'_J$ is a uniform quasigeodesic of $X$ for each maximal subinterval $J\subset I$ such that $\hat\iota\circ\gamma'_J\subset\bar X_v$. Since each connected component of the complement of each $\bar X_v$ has uniformly bounded coarse intersection with $\bar X_v$, it follows that $\hat\iota\circ\gamma$ is a uniform quasigeodesic.

\medskip

To deal with the case where the monodromy of $G$ is polynomially growing but not unipotent, we can simply perturb $\gamma$ to lie in a finite-index subgroup with unipotent monodromy and apply the above.

Finally, let $G$ be a general finitely generated free-by-cyclic group. Recall from \cref{prop:rel-hyp} that $G$ is hyperbolic relative to a finite set $\{G_1,\dots,G_k\}$ of \{f.g.\ free\}-by-cyclic subgroups with polynomially growing monodromy, and from \cref{sec:general} that $X$ is hyperbolic relative to $\{X_1,\dots,X_k\}$, where $X_i$ is given for $G_i$ by Theorem~\ref{thm:main}. As shown by the above considerations, the restriction of $\hat\iota\circ\gamma$ to each peripheral subgroup is a uniform quasigeodesic, so it follows from the distance formula for relatively hyperbolic groups, \cite[Theorem~3.1]{sisto:projections}, that $\hat\iota\circ\gamma$ is also a uniform quasigeodesic. 
% MORE PRECISE ARGUMENT:
% Indeed, $\dist_G(\gamma(s),\gamma(t))$ is coarsely equivalent to the sum of $\dist_{\hat G}(\gamma(s),\gamma(t))$ with the large distances in peripherals of $G$. By the UPG case, the distances in the $G$-peripherals are coarsely equivalent to the distances in the $X$-peripherals. Both $G$ and $X$ have the same relative structure, so the distances in $\hat G$ and $\hat X$ are coarsely equivalent. Hence the sum of $\dist_{\hat X}(\gamma(s),\gamma(t))$ with the large distances in peripherals of $X$ is coarsely equivalent to the distance in $G$. But that sum is coarsely equivalent to the distance in $X$.
\end{proof}

We remark that the equivalence of \eqref{it:1} and \eqref{it:3} in \cref{lemma:Morse-detectable}, together with the dependence of constants, is exactly the definition of \emph{Morse detectability} for $\mathrm{Cay}(G,S)$ \cite[Definition~4.17]{russellsprianotran:local}. Russell--Spriano--Tran show that every Morse detectable space is \emph{Morse local-to-global} \cite[Theorem~4.18]{russellsprianotran:local}. We therefore have the following corollary:

\begin{corollary}\label{cor:MLTG}
Every f.g.\ free-by-cyclic group is Morse detectable and thus Morse local-to-global.
\end{corollary}

% \begin{proof}
%     Suppose first that $G = F_n \rtimes_{\phi} \Z$ where $\phi$ grows polynomially of degree $\deg(\phi) = 0$. Then $G$ has a central infinite cyclic subgroup and thus it is is Morse limited by \cite[Proposition~4.6]{russellsprianotran:local}. Hence $G$ is Morse local-to-global by \cite[Lemma~4.2]{russellsprianotran:local}.

%     Suppose now that $\deg(\phi) > 0$. Then $G$ is Morse detectable by \cref{lemma:polynomially growing-Morse-detectable} and thus it is Morse local-to-global by \cite[Theorem~4.18]{russellsprianotran:local}.

%     Now suppose that $G$ is finitely generated free-by-cyclic. Then $G$ is hyperbolic relative to a collection of \{f.g.\ free\}-by-cyclic groups with polynomially growing monodromy by \cref{prop:rel-hyp}. Hence, by \cite[Theorem~5.1]{russellsprianotran:local}, $G$ is Morse local-to-global.
% \end{proof}

\subsection{Stable and strongly quasiconvex subgroups}

Let $G$ be a finitely generated group. We say that a subgroup $H \leq G$ is \emph{stable} if it is undistorted and if for every finite generating set $S$ of $G$ there exists a Morse gauge $M \colon \mathbb{R}_{\geq 0} \times \mathbb{R}_{\geq 0} \to \mathbb{R}$ such that every quasigeodesic $\gamma \colon I \to \mathrm{Cay}(G;S)$ with endpoints in $H$ is $M$-Morse.

Suppose that $G$ acts on a hyperbolic space $X$. Following \cite{balasubramanyachesserkerrmangahastrin:non}, we say that $X$ is a \emph{universal recognising space} for $G$ if, for every stable subgroup $H \leq G$, orbit maps $H \to X$ are quasi-isometric embeddings.

We observe that equivariant Morse-detectability spaces are always universal recognising spaces.

\begin{lemma} \label{lem:detectability_universal}
Suppose that an action $G\acts X$ of a group $G$ on a hyperbolic space $X$ witnesses Morse detectability of $G$. A subgroup $H<G$ is stable if and only if its orbit maps on $X$ are quasiisometric embeddings. In particular, $X$ is a universal recognising space for $G$.
\end{lemma}

\begin{proof}
If $H<G$ is stable, then since $G\acts X$ witnesses Morse detectability, every geodesic in $H$ projects to a uniform quasigeodesic of $X$, so orbit maps of $H$ on $X$ are quasiisometric embeddings. Conversely, if orbit maps of $H$ are quasiisometric embeddings, then since the map $\hat\iota \colon G \to X$ is Lipschitz, it easy to see that $H$ is undistorted in $G$. Moreover, every geodesic of $H$ maps to a quasigeodesic of $X$, and hence is Morse by the fact that $X$ witnesses Morse detectability.
\end{proof}

In particular, if $G$ is a finitely generated free-by-cyclic group and $X$ is as in Theorem~\ref{thm:main}, then $X$ is a universal recognising space for $G$. More precisely, we have the following.

\begin{corollary}\label{prop:stable-subgroups}
    Let $G$ be a f.g.\ free-by-cyclic group and $X$ be as in \cref{thm:main}. Let $H \leq G$ be a subgroup.  The following are equivalent.
    \begin{enumerate}
        \item $H \leq G$ is stable.
        \item $H$ intersects each subgroup in $\mathcal{P}$ trivially.
        \item The orbit map $j \colon H \to X$ is a quasiisometric embedding.
    \end{enumerate}
\end{corollary}

\begin{proof}
Lemma~\ref{lem:detectability_universal} gives the equivalence between the first and third items. If $H$ is stable then it must have finite, hence trivial, intersection with each product subgroup of $G$, as such subgroups are undistorted \cite[Lemma~4.1]{mutanguha:onpolynomial}. If $H$ is not stable then for each Morse gauge $M_n:(K,C)\mapsto(K+C+2)^n$ it contains a uniform quasigeodesic $\gamma_n$ that is not $M_n$--Morse. By Proposition~\ref{lemma:Morse-detectable}, the quasigeodesics $\gamma_n$ must have increasingly large (hence eventually nontrivial) intersections with cosets of elements of $\cal P$.
\end{proof}

A subgroup $H$ of a finitely generated group $G$ is said to be \emph{strongly quasiconvex} (or \emph{Morse}) if there exists some $M = M(S) >0$ such that every geodesic in $\mathrm{Cay}(G;S)$ joining points of $H$ is contained in the $M$-neighbourhood of $H$.

\begin{proposition}\label{lem:Morse-subgp-PG}
Let $G$ be a \{f.g.\ free\}-by-cyclic group with PG monodromy and let $\mathcal{P}$ be the collection of maximal product subgroups. A subgroup $H \leq G$ is strongly quasiconvex if and only if $H$ has finite index in $G$ or $H$ intersects each subgroup in $\mathcal{P}$ trivially.
\end{proposition}

\begin{proof}
If $H$ intersects each element of $\mathcal{P}$ trivially then it is stable by \cref{prop:stable-subgroups} and thus strongly quasiconvex. Moreover, every finite-index subgroup is strongly quasiconvex.

Conversely, suppose that $H \leq G$ has infinite index and that there exists some $Q \in \mathcal{P}$ such that $|Q \cap H| = \infty$. Let $G' \leq_{f.i.} G$ be a subgroup of finite index with UPG monodromy and let $H' = H \cap G'$.  Then, there exists a maximal product subgroup $P \leq G'$ such that $|P \cap H| = \infty$. We will show that $H'$ is not a strongly quasiconvex subgroup of $G'$, from which it will follow that $H$ is not strongly quasiconvex in $G$.
    
Since $P$ is a maximal product subgroup of a \{f.g.\ free\}-by-cyclic group with UPG monodromy, there exist a subgroup $G_0 \leq G'$ such that $G_0$ is \{f.g.\ free\}-by-cyclic with unipotent and linearly growing monodromy and $P \leq G_0$ (see the last paragraph of the proof of \cref{lem:cayley}). Let $H_0 := H' \cap G_0$. Since $G_0 \leq G'$ is undistorted by \cite[Lemma~4.1]{mutanguha:onpolynomial}, if we can show that $H_0 \leq G_0$ is not strongly quasiconvex, then it will follow that $H' \leq G'$ is not strongly quasiconvex, by \cite[Proposition 4.11]{tran:onstrongly}. 

Now, we argue similarly to the proof of \cite[Proposition~8.18]{tran:onstrongly}. Assume for contradiction that $H_0 \leq G_0$ is strongly quasiconvex. Since $P \leq G_0$ and $|H \cap P| = \infty$, it follows that $|H_0 \cap P| = \infty$. We claim that for every $g \in G_0$, $|H_0^g \cap P| = \infty$. 
    
Assume first that the claim is true. Then, since $H_0^g \leq G_0$ is also strongly quasiconvex and $P \leq G_0$ is an undistorted product subgroup, it must be the case that $H_0^g \cap P$ is a finite-index subgroup of $P$. Hence, for any finite subset $A \subset G_0$, we have that $\bigcap_{g\in A} H_0^g$ is infinite. However, since $H_0 \leq G_0$ is of infinite index, this directly contradicts \cite[Theorem~4.15]{tran:onstrongly}, which states that $H_0 \leq G_0$ has finite \emph{height}. Thus, $H_0 \leq G_0$ is not strongly quasiconvex.

It remains to prove the claim, assuming that $H_0 \leq G_0$ is strongly quasiconvex and $|H_0 \cap P| = \infty$. Fix $g \in G_0$ and let $T$ be the Bass--Serre tree of the splitting of $G_0$ as in \cref{prop:linear-splitting}. Then $P$ is elliptic for the action on $T$, and by maximality, thus there exists some vertex $v$ in $T$ such that $P = \mathrm{stab}_{G_0}(v)$. Let $w = g^{-1} \cdot v$ and let $v = v_0, \ldots, v_k = w$ be the vertices on the geodesic path joining $v$ to $w$ in $T$. Let $P_i = \mathrm{stab}_{G_0}(v_i)$. Each $P_i$ is a product subgroup of $G$ and $|P_i \cap P_{i+1}| = \infty$ for each $i$. It follows that if $|H_0 \cap P_i|$ is infinite then $H_0 \cap P_i$ has finite index in $P_i$. Proceeding inductively, we deduce that $H_0 \cap P_k$ has finite index in $P_k$, and thus $|H_0 \cap P^{g^{-1}}| = \infty$. Hence, $|H_0^g \cap P| = \infty$ as claimed.
\end{proof}

Combining \cref{lem:Morse-subgp-PG} with \cref{prop:rel-hyp} and \cite[Theorem~6.7]{tran:onstrongly}, we obtain the following characterisation of which undistorted subgroups of finitely generated free-by-cyclic groups are strongly quasiconvex. 

\begin{corollary}\label{cor:general-Morse-subgp}
Let $G$ be a finitely generated free-by-cyclic group and let $\mathcal{H}$ be the collection of maximal \{f.g.\ free\}-by-cyclic subgroups with PG monodromy. An undistorted subgroup $K \leq G$ is strongly quasiconvex if and only if for each $H\in\cal H$, either $K \cap H$ has finite index in $H$ or $K$ intersects every product subgroup of $H$ trivially.
\end{corollary}

\begin{remark}
Note that there is a conjectural characterisation of quasiconvex subgroups of hyperbolic \{f.g.\ free\}-by-cyclic groups (see e.g.\ \cite[Conjecture~1.1]{aimpl:rigidity}, \cite[Problem~1.5]{abdenbiwise:negative}, \cite[Conjecture~1.3]{linton:geometry}), however our methods do not seem to apply here.
\end{remark}

%%%%%%%%%%%%%%%%%%%%%%%%%%%%%%
\subsection{Morse boundary}
        
The \emph{Morse boundary} $\partial_M X$ of a proper geodesic space $X$ is the set of equivalence classes of Morse rays in $X$, where two rays $\gamma$ and $\eta$ are said to be equivalent if they lie at bounded Hausdorff distance (see \cite{cordes:morse}). The Morse boundary of a finitely generated group $G$ is defined to be the Morse boundary of $\mathrm{Cay}(G;S)$ for any finite generating set $S$ of $G$. 

A topological space $X$ is said to be \emph{$\sigma$-compact} if it is the direct limit of countably many compact spaces.  It is an \emph{$\omega$-Cantor space} if it is the direct limit of countably many Cantor spaces $X_1 \subset X_2 \subset X_3 \subset \ldots$ such that the interior of $X_i$ in $X_{i+1}$ is empty.

Charney--Cordes--Sisto showed that for any finitely generated non--virtually-free group $G$, if $\partial_M G$ is totally disconnected, $\sigma$-compact, and contains a Cantor subspace, then it is an $\omega$-Cantor space \cite[Theorem~1.4]{charneycordessisto:complete}. They used this to prove that Morse boundaries of fundamental groups of graph manifolds are $\omega$-Cantor sets. Their argument can be used to prove the following:

\begin{theorem}
Let $G$ be a \{f.g.\ free\}-by-cyclic group with PG monodromy. The Morse boundary $\partial_M G$ is an $\omega$-Cantor set.
\end{theorem}

\begin{proof}
Since $G$ is Morse local-to-global by \cref{cor:MLTG}, it follows by \cite[Theorem~A]{hesprianozbinden:sigma} that $G$ has $\sigma$-compact Morse boundary. Moreover, we may use a ping-pong type argument to construct a quasiisometrically embedded free subgroup $H$ of $G$ which does not intersect any product subgroup nontrivially and thus is stable by \cref{prop:stable-subgroups}. Hence, there is an embedding $\partial H \cong \partial_M H \to \partial_M G$, where $\partial H$ is the Gromov boundary of $H$. This shows that $\partial_M G$ contains a Cantor subspace.

Finally, by \cref{lemma:Morse-detectable}, there is a topological embedding $\partial _M G \to \partial_M X$ where $X = \mathrm{Cay}(G; S \cup \mathcal{P})$. Since $X$ is a quasitree, it follows that $\partial_M X$ and thus $\partial_M G$ is totally disconnected. The result is now given by \cite[Theorem~1.4]{charneycordessisto:complete}.
\end{proof}

\bibliographystyle{alpha}
{\footnotesize
\bibliography{bibtex}}
\end{document}